\documentclass[12pt,reqno]{amsart}
\usepackage{graphicx}
\usepackage{amsmath,amssymb,amsthm}
\usepackage{enumitem}
\usepackage[hmargin=1in,vmargin=1in]{geometry}
\usepackage[unicode,breaklinks=true,colorlinks=true]{hyperref}
\usepackage[dvipsnames]{xcolor}
\usepackage{color}
\usepackage{cite}
\usepackage{comment}
\usepackage{tikz}
\usepackage{psfrag}
\usetikzlibrary{patterns}
\usepackage{float}
\usepackage{epsfig} 

\newtheorem{thm}{Theorem}[section]
\newtheorem{lem}[thm]{Lemma}
\newtheorem{prop}[thm]{Proposition}
\newtheorem{cor}[thm]{Corollary}
\newtheorem{defn}[thm]{Definition}

\newtheorem{assum}[thm]{Assumption}

\theoremstyle{definition}
\newtheorem{eg}[thm]{Example}
\newtheorem{obs}[thm]{Remark}

\newcommand{\tnum}{\rm(\roman*)}
\newcommand{\rnum}{\rm(\alph*)}
\newcommand{\nnum}{\rm(\arabic*)}

\newcommand{\im}{\mathrm{i}}
\newcommand{\ex}{\mathrm{e}}

\newcommand{\succup}{%
      \rotatebox{90}{$\succ$}}
\newcommand{\succdn}{%
      \rotatebox{-90}{$\succ$}}

\numberwithin{equation}{section}

\def\k0{\kappa_0}

\def\bfg{{\bf{g}}}
\def\bfv{{\bf{v}}}
\def\bfw{{\bf{w}}}

\def\tbx{\tilde{\bx}}
\def\tt{\tilde{t}}

\def\bu{{\bf{u}}}
\def\bF{{\bf{f}}}

\def\bx{{\bf{x}}}
\def\b0{{\bf{0}}}
\def\Bs{B_{\rm s}}

\newcommand{\hilite}[1]{{\color{blue} #1}}
\renewcommand{\hilite}[1]{#1} 
\newcommand{\alert}[1]{{\color{red} #1}}

\begin{document}
\title[Steady states of the NSE with large Grashof numbers]
{Intrinsic expansions in large Grashof numbers for the steady states of the {N}avier--{S}tokes equations}

\author[L. Hoang]{Luan Hoang$^{1}$}
\address{$^1$Department of Mathematics and Statistics,
Texas Tech University\\
1108 Memorial Circle, Lubbock, TX 79409--1042, U. S. A.}
\email{luan.hoang@ttu.edu}

\author[M. S. Jolly]{Michael S. Jolly$^{2,*}$}
\address{$^2$Department of Mathematics\\
Indiana University\\ Bloomington, IN 47405, U. S. A.}
\email{msjolly@indiana.edu}
\thanks{$^*$Corresponding author.}

\date{\today}
\subjclass[2020]{35Q30, 35C20, 76D05}
\keywords{Navier--Stokes equations, large Grashof number,  expansion, steady states, nested spaces}
\thanks{This work was supported in part by Simons Foundation Grant MP-TSM-00002337.}
\begin{abstract}
We enable a theory of intrinsic asymptotic expansions for the steady state solutions of the full Navier--Stokes equations. Such a theory was first developed in \cite{FHJ} for Galerkin approximations. To overcome the lack of local compactness in infinite dimensional spaces, we introduce the notion of asymptotic expansions in nested spaces. When the inclusion maps between these spaces are compact, we establish the existence of such an asymptotic expansion for a subsequence of any bounded sequence. This consequently yields an intrinsic asymptotic expansion in a single normed space.
We apply this result to the steady states of the Navier--Stokes equations by utilizing the spectral fractional Sobolev spaces. In the case of the two-dimensional periodic boundary conditions, more properties relating the terms of the asymptotic expansion are obtained.  
\end{abstract}

\maketitle

\tableofcontents



\section{Introduction} \label{intro}

Kolmogorov's theory of homogeneous turbulence is most pronounced for small viscosity, large domain size and large body forces \cite{Frisch,FMRTbook}.  
\hilite{All three situations can be realized by taking the Grashof number $G$, which combines the three parameters (see \eqref{Grashof}), large. When the fluid is modeled by the Navier--Stokes equations (NSE),  the universal laws of turbulent flow emerge when physical quantities are averaged over time. As $G$ is increased, the long time dynamics of the NSE, embodied in the global attractor $\mathcal{A}$ \cite{T97}, can become very complicated.  This is typically the result of bifurcations from steady states, which remain part of the global attractor, even after becoming unstable. Such steady states can be viewed as approximations of time averaged unsteady solutions, for instance after a Hopf bifuration. With this connection to turbulence in mind, we study in this paper sequences of steady state solutions $(v_n)_{n=1}^\infty$ of  the NSE corresponding to Grashof numbers $G_n\to \infty$. }
Our results are quite different from the study of some generic properties of the set of steady state solutions in \cite{FT77,FSequil83,SaTe80}.

In our previous work \cite{FHJ} with C. Foias, we introduced a theory of asymptotic expansions for steady state solutions of the Galerkin approximation of the NSE as the Grashof number goes to infinity.
The key notion in \cite{FHJ} was a new type of asymptotic expansions for sequences in the form of a family of summations  \hilite{which are either finite sums or formal infinite series}. 
More explicitly, for a sequence $(v_n)_{n=1}^\infty$ in a normed space $(Z,\|\cdot\|_Z)$, the asymptotic expansion is
\begin{equation}\label{vxfirst}
v_n=v+\Gamma_{1,n}w_1+\Gamma_{2,n}w_2+\cdots+\Gamma_{k,n}w_k+\cdots,
\end{equation}
where all the $w_k$ are unit vectors,
the $\Gamma_{k,n}$ are positive numbers converging to zero as $n\to\infty$ and with $\Gamma_{k+1,n}$ decaying faster than $\Gamma_{k,n}$ as $n\to\infty$.
 \hilite{The asymptotic expansion \eqref{vxfirst} means that for each $k\ge 1$, the sequence $(v_n)_{n=1}^\infty$ can be approximated, as $n\to\infty$, by the sequence of partial sums
$\left(v+\sum_{j=1}^k \Gamma_{j,n}w_j\right)_{n=1}^\infty$ with which the sequence of remainders can be estimated by
$$
\lim_{n\to\infty} \frac{\left\| v_n-\left(v+\sum_{j=1}^k \Gamma_{j,n}w_j\right)\right\|_Z}{\Gamma_{k,n}}=0, 
$$
that is, the remainders go to zero, as $n\to\infty$, faster than each term in the approximate sum. In case the right-hand side of \eqref{vxfirst} is a series $v+\sum_{k=1}^\infty \Gamma_{j,n}w_k$, its summation in $k$ is \emph{not} necessarily convergent.
}
This expansion can be seen as intrinsic because its basis elements $w_k$ depend on the sequence itself. 
In that work the finite dimensional nature of the Galerkin approximation was crucial for the existence of such an asymptotic expansion.

In this paper, we develop that theory further for the steady states of the NSE itself. This is not obvious 
due to the lack of the relative compactness of bounded sets in infinite dimensional spaces.
Yet we \hilite{manage} to obtain similar results to those in \cite{FHJ}. The idea is to develop a form of asymptotic expansions in nested spaces with successive compact embeddings. The spectral  fractional Sobolev spaces based on  fractional powers of the Stokes operator play an important role in constructing such nested spaces. 
Even so, there are still differences compared to the Galerkin approximations in \cite{FHJ} and new kinds of degenerate expansions must be taken into account.

This paper is laid out so that the applications to the NSE are reached with minimal machinery; the proofs of all but the most fundamental expansion results are deferred to the Appendix.
In Section \ref{prelim}, we review the functional setting for the NSE, the Stokes operator, the bilinear mapping and other spectral  fractional Sobolev spaces.
We scale the NSE in order to put the Grashof number $G$ in front of the nonlinear terms. The advantage is the uniform boundedness of the steady states, in an appropriate norm, even when $G\to\infty$.
In Section \ref{nestedsec}, we introduce {\it strict expansions} in nested spaces in Definition \ref{uexp}. 
The exclusiveness of the three cases in this definition is verified in Proposition \ref{exclusive}.
The asymptotic uniqueness is proved in Proposition \ref{unique}.
The main result on the existence of a subsequence with such a strict expansion is Theorem \ref{mainlem}. 
Different variations of the strict expansions are introduced in Definition \ref{refinex}. 
In particular, we define {\it degenerate expansions} which can naturally arise in this infinite dimensional setting.
Theorem \ref{refinethm} shows that we can modify strict expansions to obtain a simpler form. These simpler forms still retain the meaning of asymptotic expansions but may lose their asymptotic uniqueness property.
An important consequence is Corollary \ref{maincor} which derives the asymptotic expansions for a single normed space. This result is key to our study of the  NSE. Example \ref{denexam} shows that a sequence can have both a unitary expansion and a degenerate expansion.
In Section \ref{results}, we apply the general results in Section \ref{nestedsec} to the steady states of the NSE. We also explore properties of the obtained asymptotic expansions.
The two-dimensional (2D) periodic case is studied in subsection \ref{2Dappln}. After an asymptotic expansion is obtained in Theorem \ref{2Dxp}, additional relations involving  expansion terms are given in Theorems \ref{thm1}, \ref{v0eg} and \ref{vAinvg}.
We treat the 3D case and 2D no-slip case in subsection \ref{3Dappln}. Although similar asymptotic expansions are obtained in Theorem \ref{3Dxp}, the limitations of the current estimates for the higher Sobolev norms of the solutions prevent further consideration in these cases.

\section{Preliminaries}\label{prelim}

The incompressible NSE for velocity $\bu(\bx,t)$, pressure $p(\bx,t)$ with a time-independent body force $\bF(\bx)$, for the spatial variables $\bx\in\mathbb R^d$ ($d=2,3$) and time variable $t\ge 0$, are
\begin{equation}\label{NSE}
\left\{ 
\begin{aligned}
&\frac{\partial \bu}{\partial t} -
\nu \Delta \bu  + (\bu\cdot\nabla)\bu + \nabla p = \bF , \\
&\text {div} \bu = 0.
\end{aligned}
\right.
\end{equation}
Below, we recall the standard functional setting for the NSE, see, e.g, \cite{CF88,TemamAMSbook,FMRTbook}.

\medskip
\noindent
\textit{Periodic boundary condition.} We consider the NSE with periodic boundary conditions in $\Omega =[0,L]^d$.
Let $\mathcal V$ denote the set of $\mathbb R^d$-valued $\Omega$-periodic divergence-free trigonometric polynomials with zero average over $\Omega$.
Define the spaces 
\begin{equation}\label{Hd1}
    \text{$H=$ closure of $\mathcal V$ in $L^2(\Omega)^d$ and $V=$ closure of $\mathcal V$ in $H^1(\Omega)^d$.}
\end{equation}

\medskip
\noindent\textit{No-slip boundary condition.} 
Consider the NSE in an open, connected, bounded set $\Omega$ in $\mathbb R^d$ with $C^2$-boundary $\partial \Omega$. Let $\mathbf n$ denote the outward normal vector on the boundary. 
Under the no-slip boundary condition 
$$\bu =0\text{ on }\partial \Omega ,$$
the relevant spaces are
\begin{equation}\label{Hd2}
H=\{\bu\in L^2(\Omega)^d:\ \nabla \cdot \bu=0, \quad \bu\cdot\mathbf n|_{\partial \Omega}=0\},\quad
V=\{\bu\in H_0^1(\Omega)^d:\ \nabla \cdot \bu=0 \}.
\end{equation}

\medskip
For both cases \eqref{Hd1} and \eqref{Hd2}, the inner product and its associated norm in $H$ are those of $L^2(\Omega)^d$ and are denoted by $\langle\cdot,\cdot\rangle$ and $|\cdot|$, respectively.
(We have also used $|\cdot|$ for the modulus of a vector in 
$\mathbb{C}^d$; we assume that the 
meaning will be clear from the context.)
The space $V$ is equipped with the following inner product and norm
\[
\langle\!\langle u,v \rangle\!\rangle =\int_{\Omega}\sum_{j=1}^{d} 
\frac{\partial}{\partial x_j}u(\bx)\cdot 
\frac{\partial}{\partial x_j}v(\bx)  {\rm d} \bx
\text{ and }
\|u\|=\langle\!\langle u,u \rangle\!\rangle^{1/2} \text{ for }u,v\in V.
\]
Note that the norm $\|\cdot\|$ in $V$ above is equivalent to the $H^1(\Omega)^d$-norm.

We will use the following embeddings  and identification for $H$, $V$ and their dual spaces $H'$, $V'$
\begin{equation}\label{emiden}
    V\subset H=H'\subset V'.
\end{equation}
For convenience, we define the product $\langle\cdot,\cdot\rangle_{V',V}$ on $V'\times V$ by   $$\langle w,v\rangle_{V',V}=w(v) \text{ for }w\in V', v\in V.$$
For $u\in H$ and $v\in V$, using \eqref{emiden}, we have $\langle u,v\rangle_{V',V}=\langle u,v\rangle$.

Define the bounded linear mapping $A:V\to V'$ and continuous bilinear mapping $B:V\times V\to V'$ by
\begin{equation}\label{ABdef}
    \langle Au,v\rangle_{V',V}=\langle\!\langle u,v \rangle\!\rangle, \quad 
\langle B(u,v),w\rangle_{V',V}=\int_\Omega ((u(\bx)\cdot \nabla )v(\bx))\cdot w(\bx) {\rm d} \bx,
\end{equation}
for $u,v,w\in V$.

Assume $\bF\in H$.
It is well-known that the NSE \eqref{NSE} can be written in the functional form 
\begin{equation} \label{fNSE}
\frac{{\rm d} u}{{\rm d} t}+\nu Au+B(u,u)=f \text{ in }V',
\end{equation}
with $u(t)=\mathbf u(\cdot,t)$ and $f=\mathbf f(\cdot)$ \cite{T97,FMRTbook}.
Further study of equation \eqref{fNSE} requires additional facts about the  mappings $A$ and $B$.
From \eqref{ABdef}, one has, for any $u\in V$, 
\begin{equation} \label{Auu}
\langle Au,u\rangle_{V',V}=\|u\|^2\text{ and }\|Au\|_{V'}=\|u\|.
\end{equation}
Recall the orthogonality relation of the bilinear operator $B$
(see for instance \cite{T97,FMRTbook})
\begin{equation}\label{BV}
\langle B(u,v),w\rangle_{V',V} = -\langle B(u,w),v\rangle_{V',V},
\text{ so that }  
\langle B(u,v),v\rangle_{V',V}=0\text{ for } u,v,w\in V.
\end{equation}

Let $\mathcal{P}$ be the Helmholtz-Leray projection onto $H$, 
that is, the orthogonal projection of $L^2(\Omega)^d$ onto $H$. 
Define the space $D(A)=V\cap H^2(\Omega)^d$ equipped with the norm $|Au|$, which, in fact, is equivalent to the $H^2(\Omega)^d$-norm.

The Stokes operator is the restriction of $A$ to  $D(A)$ which is a bounded linear operator from $D(A)$ onto $H$.
We have 
$$Au=-\mathcal P\Delta u \in H \text{ for } u\in D(A),$$
and 
$$\langle Au,v\rangle =\langle\!\langle u,v \rangle\!\rangle \text{ for } u\in D(A),v\in V.$$
As an unbounded linear operator on $H$, the Stokes operator  is positive, self-adjoint,  with a compact inverse, and its eigenvalues are positive numbers
$$\lambda_1\le \lambda_2\le \lambda_3\le \ldots \text{ satisfying } \lim_{n\to\infty}\lambda_n=\infty.$$
Moreover, there is an orthonormal basis $\{\varphi_n:n\in\mathbb N\}$ of $H$ for which $\varphi_n\in D(A)$ and $A\varphi_n=\lambda_n \varphi_n$ for any $n\in\mathbb N$.
Here, $\mathbb N=\{1,2,3\ldots\}$ is the set of natural numbers.
We denote 
$\kappa_0=\lambda_1^{1/2}$. 

The restriction of $B$ to $D(A)\times V$ is a continuous bilinear mapping from $D(A)\times V$ to $H$.
We have
$$
B(u,v)=\mathcal{P}\left( (u \cdot \nabla) v \right)\in H\text{ for } u\in D(A),v\in V,
$$
and, from \eqref{BV},
\begin{equation}\label{Bflip}
\langle B(u,v),w\rangle = -\langle B(u,w),v\rangle,\quad  \text{so that} \quad 
\langle B(u,v),v\rangle=0\text{ for } u,v,w\in D(A),
\end{equation}
Additionally, in the 2D periodic case, one has
\begin{equation}\label{BAzero}
\langle B(u,u),Au\rangle=0\text{ for } u\in D(A).
\end{equation}
For $u,v\in V$, we denote
$$\Bs(u,v)=B(u,v)+B(v,u).$$

Define the dimensionless Grashof number
\begin{equation}\label{Grashof}
G=\frac{1}{\nu^2\k0^{3-d/2}} \|\bF\|_{L^2(\Omega)^d}.
\end{equation}
We make the following standard change of variables, see e.g. \cite{FHJ},
\begin{equation}\label{changevar}
 \tbx = \k0 \bx,\quad \tt=\nu\k0^2 t,\quad 
 \mathbf v(\tbx,\tt)=\frac{\bu(\bx,t)}{\nu\k0 G},\quad 
q(\tbx,\tt)=\frac{p(\bx,t)}{\nu^2\k0^2G},\quad
\mathbf g(\tbx)=\frac{\bF(\bx)}{\nu^2\k0^3G}.
\end{equation}
Then $(\mathbf v,q)$ satisfies 
\begin{equation}\label{changesys}
\left\{ 
\begin{aligned}
&\frac{\partial \mathbf v}{\partial \tilde t} - \Delta_{\tbx} \mathbf v  + G (\mathbf v\cdot\nabla_{\tbx})\mathbf v + \nabla_{\tbx} q = \mathbf g,\\
&\text {div}_{\tbx} \mathbf v = 0,
\end{aligned}
\right.
\end{equation}
for
\begin{equation}\label{newdom} 
\tbx\in \tilde \Omega:=\kappa_0 \Omega,\quad \tt>0,
\end{equation}
and $\mathbf v$ is subject to the same \hilite{type of} boundary conditions as $\bu$.

Hereafter, we suppress the tildes in \eqref{changesys} and \eqref{newdom}.
As done for \eqref{fNSE}, we can rewrite \eqref{changesys} in the following functional form
\begin{equation}\label{sc}
\frac{{\rm d} v}{{\rm d} t}+Av+ GB(v,v) = g,
\end{equation}  
with $v(t)=\mathbf v(\cdot,t)$ and $g=\mathbf g(\cdot)$.
Note that the function $g$ in \eqref{sc} satisfies 
\begin{equation}\label{g1}
    \|g\|_{L^2(\Omega)^d}=1.
\end{equation} 
Moreover, the Stokes operator $A$ in \eqref{sc} has the first eigenvalue 
\begin{equation}\label{unitlam}
\lambda_1 = 1.    
\end{equation}

We now recall the fractional operators and spaces.
For $\alpha\ge 0$, define
$$ A^\alpha u =\sum_{n=1}^\infty \lambda_n^\alpha c_n \varphi_n \text{ for }
u=\sum_{n=1}^\infty c_n \varphi_n\in D(A^\alpha),$$
where 
$$ D(A^\alpha)=\left \{ u=\sum_{n=1}^\infty c_n \varphi_n\in H: \sum_{n=1}^\infty \lambda_n^{2\alpha} |c_n|^2<\infty\right\}$$ 
equipped with the norm 
$\|u\|_{D(A^\alpha)}=|A^\alpha u|.$
We have $D(A^0)=H$, $\|u\|_{D(A^0)}=|u|$, $A^0={\rm Id}_H$, $D(A^{1/2})=V$, $\|u\|_{D(A^{1/2})}=\|v\|=|A^{1/2}u|$ and $D(A^1)=D(A)$, $\|u\|_{D(A^1)}=|Au|$, $A^1=A$.
Thanks to \eqref{unitlam}, we have
$$\|u\|_{D(A^\alpha)}\ge \|u\|_{D(A^\beta)}  \text{ for }\alpha\ge \beta\ge 0, \ u\in D(A^\alpha),$$
which particularly implies 
\begin{equation}\label{Poincare}
|Au|\ge \|u\|  \text{ for  }u\in D(A),\quad  \|u\|\ge |u| \text{ for  }u\in V.
\end{equation}
Finally, for any numbers $\alpha>\beta\ge 0$,  the space $D(A^\alpha)$ is compactly embedded into $D(A^\beta)$  \cite[Section 2.2]{TemamSIAMbook}.

\section{Asymptotic expansions in nested spaces and a consequence}\label{nestedsec}

In this section, we introduce new types of asymptotic expansions in nested spaces.
We investigate their properties and derive a consequence that establishes an asymptotic expansion in a single normed space.

\subsection{Asymptotic expansions in nested spaces}\label{xnested}

\begin{defn}\label{uexp}
Let $(Z_k,\|\cdot\|_{Z_k})$ for $k\ge 0$ be normed spaces over $\mathbb K=\mathbb C$ or $\mathbb R$ such that
\begin{equation}\label{nested}
Z_0\subset Z_1 \subset Z_2 \subset \ldots \subset Z_k \subset Z_{k+1} \subset\ldots  
\end{equation}
with continuous embeddings. Denote $\mathcal Z=(Z_k)_{k=0}^\infty$.

    We say a sequence $(v_n)_{n=1}^\infty$ in $Z_0$  has a \emph{strict expansion} in $\mathcal Z$ if it satisfies one of the following three conditions.
\begin{enumerate}[label={\rm (\Roman*)}]
    \item\label{d1} There exists $v\in Z_0$ such that  $v_n=v$ for all $n\ge 1$.

    \item\label{d2} There exist $K\in \mathbb{N}$,  $v\in Z_0$,  $w_k\in Z_k$,  $w_n^{(k)}\in Z_{k-1}$ and $\Gamma_{k,n}>0$, for $n\ge 1$ and $1\le k\le K$, such that 
\begin{align}
\label{b0}
\lim_{n\to\infty} \Gamma_{1,n}&=0,\\
\label{b1}
\lim_{n\to\infty} \frac{\Gamma_{k+1,n}}{\Gamma_{k,n}}&=0 
\text{ for all   $1\le k< K$,}
\end{align}
\begin{align}
\label{b3}
\lim_{n\to\infty} \|w_n^{(k)}-w_k\|_{Z_k}&=0  \text{ for all $1\le k\le K$,}\\
\label{b2}
\|w_n^{(k)}\|_{Z_{k-1}}&=1  \text{ for all $n\ge 1$, $1\le k\le K$,}
\end{align}
\begin{equation}\label{b4}
    v_n=v+\sum_{j=1}^{k-1} \Gamma_{j,n}w_j+\Gamma_{k,n}w_n^{(k)} \text{ for all $n\ge 1$, $1\le k\le K$, }
\end{equation}
and
\begin{equation}\label{b6}
    w_n^{(K)}=w_{K} \text{ for all $n\ge 1$.}
\end{equation}

   \item\label{d3}  There exist $v\in Z_0$,  $w_k\in Z_k$, $w_n^{(k)}\in Z_{k-1}$, $N_k \in \mathbb{N}$  and $\Gamma_{k,n}> 0$ for all $n\ge 1$, $k\ge 1$  such that one has \eqref{b0}, while  the limits in  \eqref{b1} and \eqref{b3} hold for all $k\ge 1$, the equation in \eqref{b4} holds for all $n,k\ge 1$, and the unit vector condition in \eqref{b2} holds for all $k\ge 1$, $n\ge N_k$.
\end{enumerate}

We say the expansion in Case \ref{d1} is trivial, the expansions in Cases \ref{d1} and \ref{d2} are finite, and 
the expansion in Case \ref{d3} is infinite.
\end{defn}

We denote the strict expansion in Definition \ref{uexp} by  
\begin{equation}\label{vsum}
    v_n\approx v+\sum_{k\in\mathcal N} \Gamma_{k,n} w_k,
\end{equation}
where the set $\mathcal N$ is empty in Case \ref{d1}, is $\{1,2,\ldots,K\}$ in Case \ref{d2}, and is $\mathbb N$ in Case \ref{d3}.

\begin{obs}\label{remscl}
Assume the strict expansion \eqref{vsum}.
\begin{enumerate}[label=\rnum]
\item If $\mathcal N\ne \emptyset$, then it follows from \eqref{b0} and the limit in \eqref{b1} that 
\begin{equation}\label{Gamze}
\lim_{n\to\infty} \Gamma_{k,n}=0 \text{ for all }k\in\mathcal N.
\end{equation}

\item We claim that 
\begin{equation}\label{vlim}
    \lim_{n\to\infty} \|v_n-v\|_{Z_0}=0.
\end{equation}

Indeed, \eqref{vlim} is obviously true in Case \ref{d1} of Definition \ref{uexp}.
In Cases \ref{d2} and \ref{d3},  by taking the limit in $Z_0$ as $n\to\infty$ of equation \eqref{b4} for $k=1$  and using \eqref{b0} and property \eqref{b2} for $k=1$, one obtains \eqref{vlim} again.

\item The justification for the terminology ``expansion" is the following remainder estimate
\begin{equation}\label{aapr}
  \frac{  \left \|  v_n-\left(v+\sum_{j=1}^{k-1} \Gamma_{j,n}w_j\right)\right\|_{Z_k}}{\Gamma_{k-1,n}}
  =\frac{\Gamma_{k,n}}{\Gamma_{k-1,n}}\| w_n^{(k)}\|_{Z_k}\to 0\cdot \|w_k\|_{Z_k}=0\text{ as }n\to\infty. 
\end{equation}
\hilite{We also emphasize that in the case \ref{d3}, just as with (1.1), there is no information about the convergence, in $k$, of the formal series $v+\sum_{k=1}^\infty \Gamma_{k,n} w_k$.}

\item In Case \ref{d2}, because $w_K=w_n^{(K)}$ and $\|w_n^{(K)}\|_{Z_{K-1}}=1$, one has
\begin{equation}\label{wKnonz}
    w_K\ne 0,
\end{equation}
that is the last term in the strict expansion is nonzero.
We also have from \eqref{b6} with $k=K$ that
\begin{equation}\label{finvn}
    v_n=v+\sum_{k=1}^K \Gamma_{k,n}w_k\text{ for } n\ge 1.
\end{equation}
Hence, for $1\le k\le K$ and $n\ge 1$, 
\begin{equation}\label{wnkfin}
    w_n^{(k)}=w_k+\sum_{j=k+1}^K \frac{\Gamma_{j,n}}{\Gamma_{k,n}}w_j.
\end{equation}

\item In the case $\mathcal N\ne \emptyset$, we denote, for convenience,
\begin{equation}\label{Nbar}
 \bar N_k=\begin{cases}
     1, & \text{ in Case \ref{d2} for } 1\le k\le K,\\
     \max\{N_1,N_2,\ldots,N_k\}, & \text{ in Case \ref{d3} for all } k\ge 1.
 \end{cases}
 \end{equation}
Then the unit vector property \eqref{b2} holds true for $k\ge 1$ and $n\ge \bar N_k$.

\item Any subsequence $(v_{n_j})_{j=1}^\infty$ of $(v_n)_{n=1}^\infty$ 
has the strict expansion
\begin{equation}\label{vsubsum}
   v_{n_j}\approx v+\sum_{k\in\mathcal N} \Gamma_{k,n_j} w_k
\end{equation}
which is of the same case \ref{d1}, \ref{d2} or \ref{d3} as  $(v_n)_{n=1}^\infty$.
\end{enumerate}
\end{obs}

We investigate the exclusiveness of the three cases in Definition \ref{uexp} and the uniqueness of the strict expansion \eqref{vsum}. First, we need the following preliminary uniqueness property.

\begin{lem}\label{preuniq}
Let $K$ be in $\mathbb N$  and $(Z_k,\|\cdot\|_{Z_k})$ for $0\le k \le K$ be normed spaces over $\mathbb K=\mathbb C$ or $\mathbb R$.
Let $(v_n)_{n=1}^\infty$ be a sequence in $Z_0$. Suppose there are vectors $v\in Z_0$,  $w_k\in Z_k$ and increasing integers $N_k\ge 1$ for $1\le k\le K$,  vectors $w_n^{(k)}\in Z_{k-1}$ and positive numbers $\Gamma_{k,n}$ for $n\ge 1$ and $1\le k\le K$, such that \eqref{b3}, \eqref{b4} hold, the limit in \eqref{Gamze} holds for $1\le k\le K$ and
\begin{equation*}
    \|w_n^{(k)}\|_{Z_{k-1}}=1  \text{ for all $1\le k\le K$, $n\ge N_k$.}
\end{equation*}
Then $v$, $w_k$, $w_n^{(k)}$ and $\Gamma_{k,n}$  are uniquely determined for $1\le k\le K$ and $n\ge N_k$.
\end{lem}
\begin{proof}
In this case, we still have the limit \eqref{vlim}. Hence the vector $v$ is determined uniquely.

For $n\ge N_1$, one has 
$$\|v_n-v\|_{Z_0}=\Gamma_{1,n}\|w_n^{(1)}\|_{Z_0}=\Gamma_{1,n},$$ 
which implies $\Gamma_{1,n}$ is uniquely determined. Then
$w_n^{(1)}=(v_n-v)/\Gamma_{1,n}$ is uniquely determined and, subsequently, 
$w_1=\lim_{n\to\infty} w_n^{(1)}$  in $Z_1$ is uniquely determined.

Recursively, for $1\le k<K$, suppose $w_1,\ldots,w_{k}$,  and  $w_n^{(j)}$, $\Gamma_{j,n}$, for $1\le j\le k$ and $n\ge N_j$, are already uniquely determined. 
For $n\ge N_{k+1}$, the numbers $\Gamma_{j,n}$ for $1\le j\le k$ are already determined uniquely, and we have
$$\left\|v_n-\sum_{j=1}^{k} \Gamma_{j,n}w_j\right \|_{Z_k}=\Gamma_{k+1,n}\|w_n^{(k+1)}\|_{Z_k}=\Gamma_{k+1,n}$$
which implies $\Gamma_{k+1,n}$ is uniquely determined. Consequently, the vector $w_n^{(k+1)}$, for $n\ge  N_{k+1}$, is uniquely determined by 
$$ w_n^{(k+1)} = \frac1{\Gamma_{k+1,n}}\left(v_n-\sum_{j=1}^{k} \Gamma_{j,n}w_j\right),$$
and then $w_{k+1}=\lim_{n\to\infty} w_n^{(k+1)}$ in $Z_{k+1}$ is uniquely determined.    
\end{proof}

The following exclusiveness was stated as an obvious fact in \cite[Defintion 4.1]{FHJ}. In fact, it can be justified implicitly by the proof of \cite[Proposition 4.1]{FHJ}. In Proposition \ref{exclusive} below, we present a complete proof.

\begin{prop}\label{exclusive}
The three cases \ref{d1}, \ref{d2}, \ref{d3} in Definition \ref{uexp} are exclusive.
\end{prop}
\begin{proof}
Thanks to \eqref{vlim}, the vector $v$ is uniquely determined.
Suppose we have Case \ref{d2} or \ref{d3} in Definition \ref{uexp}. 
 Note from \eqref{Nbar} that  $\bar N_1=1$ in Case \ref{d2} and $\bar N_1=N_1$ in Case \ref{d3}. 
For $n\ge \bar N_1$, one has
$$\|v_n-v\|_{Z_0}=\Gamma_{1,n}\|w_n^{(1)}\|_{Z_0}=\Gamma_{1,n}>0,$$ which implies that we cannot have Case \ref{d1}.

Now, suppose we have both Cases \ref{d2} and \ref{d3} concurrently, that is, we have both a nontrivial, finite strict expansion, see \eqref{finvn},
\begin{equation}\label{v2} 
v_n=v+\sum_{k=1}^K \Gamma_{k,n}'w_k' \text{ for } n\ge 1,
\end{equation}
and an infinite strict expansion 
\begin{equation} \label{v3} 
v_n\approx v+\sum_{k=1}^\infty \Gamma_{k,n}w_k\text{ together with integers $N_k$ and vectors $w_n^{(k)}$.}
\end{equation}
By the virtue of Lemma \ref{exclusive},
we must have 
\begin{equation}\label{v4}
w_k=w_k'\text{ and }\Gamma_{k,n}=\Gamma_{k,n}'\text{ for $1\le k\le K$ and $n\ge \bar N_k$.}
\end{equation}
For $n\ge \bar N_{K+1}$, we have from \eqref{v3} that
\begin{equation}\label{wcontra} 
\left\|v_n-v-\sum_{k=1}^{K} \Gamma_{k,n}w_k\right\|_{Z_k}=\Gamma_{K+1,n}\|w_n^{(K+1)}\|_{Z_K}=\Gamma_{K+1,n}>0,
\end{equation}
which contradicts \eqref{v2} and \eqref{v4}.
In conclusion, the three cases \ref{d1}, \ref{d2}, \ref{d3} in Definition \ref{uexp} are exclusive.
\end{proof}

On the issue of uniqueness of the strict expansions, we have the following ``asymptotic uniqueness" property.

\begin{prop}\label{unique}
Assume a sequence $(v_n)_{n=1}^\infty$ has a strict expansion \eqref{vsum} as in Definition \ref{uexp}.
 Let $\bar N_k$ be defined by \eqref{Nbar}.
 Then the set $\mathcal N$, vector $v$, 
 and, in the cases \ref{d2} and \ref{d3},
 the vectors $w_k$  for $k\in \mathcal N$, 
the vectors   $w_n^{(k)}$ and positive numbers $\Gamma_{k,n}$ for $k\in \mathcal N$ and $n\ge \bar N_k$, are uniquely determined.
 \end{prop}

 The proof of Proposition \ref{unique} is given in the Appendix.

\begin{obs}\label{caution}
    There is a subtlety in the ``asymptotic" part of the uniqueness in Proposition \ref{unique} for Case \ref{d3}. It depends on the \emph{given} integers $N_k$. Assume we have another infinite strict expansion
    \begin{equation}\label{another}
    v_n\approx v+\sum_{k=1}^\infty \Gamma'_{k,n}w'_k\text{ together with vectors $w_n'^{(k)}$ and claimed numbers  $N'_k$.}
    \end{equation}
    Define $\bar N_k$ and $\bar N'_k$ correspondingly.
    Then $\Gamma_{k,n}=\Gamma'_{k,n}$ and $w_n^{(k)}=w_n'^{(k)}$ for $n\ge \max\{\bar N_k,\bar N'_k\}$. Suppose, say, $\bar N_k>\bar N'_k$. Although $\Gamma'_{k,n}$ and $w_n'^{(k)}$ are uniquely determined according to the strict expansion \eqref{another} for $n\ge \bar N'_k$, it does not imply $\Gamma_{k,n}'=\Gamma_{k,n}$ and $w_n'^{(k)}=w_n^{(k)}$ for 
    $\bar N'_k< n<\bar N_k$.
\end{obs}

The following is the main abstract expansion existence result which is a counterpart of the finite dimensional version \cite[Lemma 4.2]{FHJ}. Its proof follows the lines of that of \cite[Lemma 4.2]{FHJ} with careful modifications to deal with different norms in different infinite dimensional spaces.

\begin{thm}\label{mainlem} 
Let  $\mathcal Z=((Z_k,\|\cdot\|_{Z_k}))_{k=0}^\infty$ be a family of normed spaces over $\mathbb K=\mathbb C$ or $\mathbb R$ such that \eqref{nested} holds with all embeddings being compact. Then any convergent sequence in $Z_0$ has a subsequence that possesses a strict expansion in $\mathcal Z$.
\end{thm}

\begin{proof}   
Let $(v_n)_{n=1}^\infty$ be a sequence in $Z_0$ that converges to $v\in Z_0$.
Denote $\varphi_0(n)=n$ for $n\in\mathbb N$.
Then 
\begin{equation}\label{vnlim}
    \lim_{n\to \infty}\| v_{\varphi_0(n)}-v\|_{Z_0}=0.
\end{equation} 

If there is $N_0\ge 1$ such that $v_n=v$ for all $n\ge N_0$, then we have Case \ref{d1} in Definition \ref{uexp} for a subsequence of $(v_n)_{n=1}^\infty$. 

Otherwise, there is a subsequence $(\varphi_1(n))_{n=1}^\infty$ of $(\varphi_0(n))_{n=1}^\infty$ such that
$$v_{\varphi_1(n)}\ne v\text{ for all }n\in\mathbb N.$$
Define $\gamma_{\varphi_1(n)}^{(1)}=\|v_{\varphi_1(n)}-v\|_{Z_0}$ for all $n\in \mathbb N$. 
For $n\in \mathbb N$, we have $\gamma_{\varphi_1(n)}^{(1)}>0$, and can define 
\begin{equation*}
w_{\varphi_1(n)}^{(1)}=\frac{1}{\gamma_{\varphi_1(n)}^{(1)}}(v_{\varphi_1(n)}-v).    
\end{equation*}
For all $n\in \mathbb N$, defining $\Gamma_{1,{\varphi_1(n)}}=\gamma_{\varphi_1(n)}^{(1)}>0$, one has
\begin{equation}\label{w2}
    v_{\varphi_1(n)}=v+\Gamma_{1,{\varphi_1(n)}}w_{\varphi_1(n)}^{(1)}, \
    \|w_{\varphi_1(n)}^{(1)}\|_{Z_0}=1,\text{ and by \eqref{vnlim}, } 
    \lim_{n\to\infty}\Gamma_{1,{\varphi_1(n)}}=0.
\end{equation}
By the compact embedding of $Z_0$ into $Z_1$, there exists a 
subsequence of $(\varphi_1(n))_{n=1}^\infty$, still denoted by $(\varphi_1(n))_{n=1}^\infty$, and $w_1\in Z_1$  such that 
\begin{equation}\label{w3}
    \lim_{n\to\infty}\|w_{\varphi_1(n)}^{(1)}-w_1\|_{Z_1}=0.
\end{equation}

We will construct $\varphi_k(n)$, $\Gamma_{k,\varphi_k(n)}$, $w_k$, $w_{\varphi_k(n)}^{(k)}$ recursively in $k$ as follows.
Denote $w_0=v$, let $w_n^{(0)}=v_n$ and $\Gamma_{0,n}=1$ for all $n\ge 1$.
For $k\ge 1$, we define the following statement ($T_k$).

\medskip
($T_k$) \emph{There are subsequences $(\varphi_j(n))_{n=1}^\infty$ of $(n)_{n=1}^\infty$, for $1\le j\le k$, with each $(\varphi_{j}(n))_{n=1}^\infty$ being a subsequence of $(\varphi_{j-1}(n))_{n=1}^\infty$,  vectors $w_j\in Z_j$, positive numbers $\Gamma_{j,{\varphi_j(n)}}$  for $n\in\mathbb N$, and vectors $w_{\varphi_k(n)}^{(k)}\in Z_{k-1}$ for $n\in\mathbb N$, 
such that one has
\begin{equation}\label{w4}
    \lim_{n\to\infty}\frac{\Gamma_{j,{\varphi_j(n)}}}{\Gamma_{j-1,{\varphi_j(n)}}}=0 \text{ for }1\le j\le k,
\end{equation}
\begin{equation}\label{w5}
    \|w_{\varphi_k(n)}^{(k)}\|_{Z_{k-1}}=1 \text{ for  all } n\in \mathbb N,
\end{equation}
\begin{equation}\label{w6}
    \lim_{n\to\infty}\|w_{\varphi_k(n)}^{(k)}-w_k\|_{Z_k}=0,
\end{equation}
and 
\begin{equation}\label{w7}
 v_{\varphi_k(n)}=v+\Gamma_{1,{\varphi_k(n)}}w_1+\Gamma_{2,{\varphi_k(n)}}w_2+\cdots+
\Gamma_{k-1,{\varphi_k(n)}}w_{k-1}+\Gamma_{k,{\varphi_k(n)}}w_{\varphi_k(n)}^{(k)}
\end{equation}
for all $n\in\mathbb N$.
} 

\medskip
By \eqref{w2} and \eqref{w3}, we already have ($T_1$).
Let $k\ge 1$. Assume all ($T_j$) holds true for $1\le j\le k$. Consider the sequence $(w_{\varphi_k(n)}^{(k)})_{n=1}^\infty$ in \eqref{w7} of ($T_k$).

\medskip
\emph{Case 1. $w_{\varphi_k(n)}^{(k)}=w_k$ for all large $n\in \mathbb N$.} 
Set $K=k$.
Then there is a subsequence $(\varphi(n))_{n=1}^\infty$ of $(\varphi_K(n))_{n=1}^\infty$ such that
\begin{equation*}
    w_{\varphi(n)}^{(K)}= w_{K}\text{ for all }n\in\mathbb N.
\end{equation*}

For $n\in\mathbb N$, let $V_n=v_{\varphi(n)}$, and, for $1\le j\le K$, $\widetilde \Gamma_{j,n}=\Gamma_{j,\varphi(n)}$ and $\widetilde w_n^{(j)}=w_{\varphi(n)}^{(j)}$. Then $(V_n)_{n=1}^\infty$ is a subsequence of $(v_{\varphi_j(n)})_{n=1}^\infty$, for $1\le j\le K$, and in particular, of $(v_n)_{n=1}^\infty$. 
One has, from \eqref{w2},
\begin{equation}\label{wtil1}
    \lim_{n\to\infty} \widetilde\Gamma_{1,n}=0,
\end{equation}
 from \eqref{w4} of ($T_K$)
\begin{equation}\label{wtil4}
    \lim_{n\to\infty} \frac{\widetilde\Gamma_{j+1,n}}{\widetilde\Gamma_{j,n}}=0\text{ for }1\le j<K,
\end{equation}
and from \eqref{w5},  \eqref{w6}, \eqref{w7} of ($T_j$), for $1\le j\le K$ and $n\in \mathbb N$,
\begin{equation}\label{wtil5}
    \|\widetilde w_n^{(j)}\|_{Z_{j-1}}=1, 
\end{equation}
\begin{equation}\label{wtil6}
    \lim_{n\to\infty}\|\widetilde w_n^{(j)}-w_j\|_{Z_j}=0,
\end{equation}
\begin{equation}\label{wtil7}
 V_n=v+\widetilde\Gamma_{1,n}w_1+\widetilde\Gamma_{2,n}w_2+\cdots+
\widetilde\Gamma_{j-1,n}w_{j-1}+\widetilde\Gamma_{j,n}\widetilde w_n^{(j)}.
\end{equation}

Therefore, we obtain Case \ref{d2} in Definition \ref{uexp} for the sequence $(V_n)_{n=1}^\infty$ 
with $\widetilde \Gamma_{1,n}$ replacing $ \Gamma_{1,n}$ and  $\widetilde w_n^{(j)}$ replacing $w_n^{(j)}$. We stop the recursion.

\medskip
\emph{Case 2. There is a subsequence $(\varphi_{k+1}(n))_{n=1}^\infty$ of $(\varphi_k(n))_{n=1}^\infty$ such that
$$w_{\varphi_{k+1}(n)}^{(k)}\ne w_k\text{ for all } n\in\mathbb N.$$}

Define, for all $n\in\mathbb N$,  
$$\gamma_{\varphi_{k+1}(n)}^{(k+1)}=\|w_{\varphi_{k+1}(n)}^{(k)}-w_k\|_{Z_k}>0, \quad 
\Gamma_{k+1,\varphi_{k+1}(n)}=\Gamma_{k,\varphi_{k+1}(n)}\gamma_{\varphi_{k+1}(n)}^{(k+1)},$$
and 
$$w_{\varphi_{k+1}(n)}^{(k+1)}=\frac1{\gamma_{\varphi_{k+1}(n)}^{(k+1)}}(w_{\varphi_{k+1}(n)}^{(k)}-w_k).
$$
The last identity yields
\begin{equation}\label{w9}
w_{\varphi_{k+1}(n)}^{(k)}=w_k+\gamma_{\varphi_{k+1}(n)}^{(k+1)}w_{\varphi_{k+1}(n)}^{(k+1)}.    
\end{equation}
Thanks to ($T_k$), we can substitute  the last equation \eqref{w9} into \eqref{w7} with ${\varphi_{k+1}(n)}$ replacing ${\varphi_{k}(n)}$ to obtain  
\begin{equation*}
  v_{\varphi_{k+1}(n)}=v+\Gamma_{1,{\varphi_{k+1}(n)}}w_1+\Gamma_{2,{\varphi_{k+1}(n)}}w_2+\cdots+
\Gamma_{k,\varphi_{k+1}(n)}w_{k}+\Gamma_{k+1,{\varphi_{k+1}(n)}}w_{\varphi_{k+1}(n)}^{(k+1)}
\end{equation*}
for all $n\in\mathbb N$.
It is also clear that $\|w_{\varphi_{k+1}(n)}^{(k+1)}\|_{Z_k}=1$ for all $n\in \mathbb N$.
By \eqref{w6}, 
$$\lim_{n\to\infty }\gamma_{\varphi_{k+1}(n)}^{(k+1)}=0,\text{ hence, }
\lim_{n\to\infty}\frac{\Gamma_{k+1,{\varphi_{k+1}(n)}}}{\Gamma_{k,{\varphi_{k+1}(n)}}}=\lim_{n\to\infty }\gamma_{\varphi_{k+1}(n)}^{(k+1)}=0.$$

By the compact embedding of $Z_k$ into $Z_{k+1}$ again, there exist $w_{k+1}\in Z_{k+1}$ and a subsequence of $(\varphi_{k+1}(n))_{n=1}^\infty$, but still denoted by $(\varphi_{k+1}(n))_{n=1}^\infty$,  such that
\begin{equation*}
    \lim_{n\to\infty}\|w_{\varphi_{k+1}(n)}^{(k+1)}-w_{k+1}\|_{Z_{k+1}}=0.
\end{equation*}
Thus, ($T_{k+1}$) holds.

If the recursion  does not stop at any step $k$, then  we have ($T_k$) holds for all $k\ge 1$. 
Let $\varphi(n)=\varphi_n(n)$ and   $V_n=v_{\varphi(n)}$ for $n\ge 1$.
We define positive numbers $\widetilde\Gamma_{k,n}$ and vectors $\widetilde w_n^{(k)}\in Z_{k-1}$ recursively as follows.

For $k=1$, define 
$$\widetilde\Gamma_{1,n}=\Gamma_{1,\varphi(n)}=\gamma_{\varphi(n)}^{(1)}\text{ and } \widetilde w_n^{(1)}=w_{\varphi(n)}^{(1)}\text{  for all }n\ge 1.$$

For $k\ge 2$, when $1\le n <k$ let
$$z_{k,n}=V_n-v-\sum_{j=1}^{k-1}\widetilde\Gamma_{j,n}w_j$$
and define 
\begin{equation}\label{GWsmall}
\left\{     \begin{aligned}
&\widetilde\Gamma_{k,n}=\|z_{k,n}\|_{Z_{k-1}}, \quad \widetilde w_n^{(k)}=\frac{1}{\widetilde\Gamma_{k,n}}z_{k,n} &&\text{ in the case }z_{k,n} \ne 0,\\
&\widetilde\Gamma_{k,n}=\frac1{2^{nk}} \widetilde\Gamma_{k-1,n}, \quad \widetilde w_n^{(k)}=0 &&\text{ in the case }z_{k,n} = 0;
\end{aligned}
\right.
\end{equation}
while when $n\ge k$, define
\begin{equation}\label{GWlarge}
\widetilde\Gamma_{k,n}=\Gamma_{k,\varphi(n)} \text{ and }
\widetilde w_n^{(k)}=w_{\varphi(n)}^{(k)}.
\end{equation}

Note that we still have the limit \eqref{wtil1}, while the limits \eqref{wtil4} and \eqref{wtil6} are true for all $j\ge 1$.
Set $N_k=k$ for all $k\ge 1$. For $j\ge 1$, by \eqref{GWlarge} and property \eqref{w5} of ($T_j$), we have \eqref{wtil5}  for all $n\ge N_j$.
Moreover, \eqref{wtil6} holds for all $j\ge 1$.

By \eqref{w7} and \eqref{GWlarge}, we obtain \eqref{wtil7} for $j\ge 1$, $n\ge N_j$.
By definition \eqref{GWsmall}, we also obtain \eqref{wtil7} for $n< N_j$. Thus, the identity \eqref{wtil7} holds true for all $j\ge 1$ and $n\ge 1$.
Therefore, we obtain Case \ref{d3} in Definition \ref{uexp} for the sequence $(V_n)_{n=1}^\infty$ 
with $\widetilde\Gamma_{k,n}$ replacing $\Gamma_{k,n}$ and  $\widetilde w_n^{(k)}$ replacing $w_n^{(k)}$.
\end{proof}

\begin{defn}\label{refinex}
Let  $\mathcal Z=(Z_k)_{k=0}^\infty$ be as in Definition \ref{uexp} and $(v_n)_{n=1}^\infty$ be a sequence in $Z_0$.
\begin{enumerate}[label=\tnum]
    \item   We say $(v_n)_{n=1}^\infty$ has a \emph{relaxed expansion} if it satisfies the same conditions in Definition \ref{uexp} except that the unit vector condition \eqref{b2}  is removed, and, 
    \begin{equation}\label{wKcond}
        w_K\ne 0\text{ in Case \ref{d2}.}
    \end{equation}

    \item\label{uni}   We say $(v_n)_{n=1}^\infty$ has a \emph{unitary expansion} if it has a relaxed expansion and in Case \ref{d2}, respectively, Case \ref{d3}, it requires that
    \begin{equation}\label{wkone}
        \|w_k\|_{Z_k}=1\text{ for all $1\le k\le K$, respectively,  $k\ge 1$.}
    \end{equation}
    
\item\label{degenx}  We say $(v_n)_{n=1}^\infty$ has a \emph{degenerate expansion} if it satisfies the same conditions as an infinite unitary expansion in part \ref{uni} except that the requirement \eqref{wkone} is replaced with the following: either 
    \begin{enumerate}[label=\rnum]
        \item $w_k=0$ for all $k\ge1$, or 
        \item there exists an integer $N\ge 1$ such that
            \begin{equation}\label{wktwo}
                \|w_k\|_{Z_k}=1\text{ for all $1\le k\le N$, and $w_k=0$ for all $k>N$.}
            \end{equation}
    \end{enumerate} 
\end{enumerate}  
\end{defn}

We will use the same notation in \eqref{vsum} to denote various expansions in Definition \ref{refinex} but will clearly specify its type before it is used.
In all definitions in Definition \ref{refinex}, it is clear that the existence of the numbers $N_k$ is removed.
Moreover, in the case of relaxed expansions, we still have the approximation property \eqref{aapr}, and hence these expansions are meaningful.

By using property \eqref{wKnonz} to verify condition \eqref{wKcond}, it is clear that 
\begin{equation}\label{srex}
    \text{any strict expansion in $\mathcal Z$ is also a relaxed expansion in $\mathcal Z$.}
\end{equation}

\begin{thm}\label{refinethm}
Assume a sequence $(v_n)_{n=1}^\infty$ has a relaxed expansion \eqref{vsum} in the sense of Definition \ref{refinex}.
Then there exists a subsequence $\tilde {\mathcal Z}=(\tilde Z_k)_{k=0}^\infty$ of  $\mathcal Z$ such that 
the sequence $(v_n)_{n=1}^\infty$ has either a unitary expansion or a degenerate expansion in $\tilde{\mathcal Z}$.
\end{thm}

The main idea is a restructuring process where we remove certain zero terms (those not in a tail of zeros) and normalize the nonzero terms other than $v$. We provide a detailed proof in the Appendix.
\hilite{We emphasize that the subsequence $\tilde {\mathcal Z}$ and the expansion in $\tilde {\mathcal Z}$ may not be unique, even though they are unique if they are obtained via the restructuring process (see Example \ref{denexam} below).}

\begin{obs}\label{degenrmk}
Assume \eqref{vsum} is a degenerate expansion in $\mathcal Z=(Z_k)_{k=1}^\infty$.
For convenience, denote $w_0=v$, $\Gamma_{0,n}=1$, and, in Definition \ref{refinex}\ref{degenx}(a), let $N=0$.
Then we have $N\ge 0$ for both (a) and (b) of Definition \ref{refinex}\ref{degenx}.
Define the remainders
\begin{equation}\label{Rmn}
    R_{m,n}=v_n-\sum_{k=0}^m \Gamma_{k,n} w_k\text{ for }m\ge 0,n\ge 1.
\end{equation}

    Let $m\ge N$. For $n\ge 1$, we have
    $$\|R_{N,n}\|_{Z_{m+1}}=  \|R_{m,n}\|_{Z_{m+1}} =\Gamma_{m+1,n}\|w_n^{(m+1)}\|_{Z_{m+1}}.$$    
It follows that 
\begin{equation}\label{denremest}
\frac{\|R_{N,n}\|_{Z_{m+1}}}{\Gamma_{m+1,n}}=\|w_n^{(m+1)}\|_{Z_{m+1}}\to \|w_{m+1}\|_{Z_{m+1}}=0 \text{ as }n\to\infty.
\end{equation}  
\end{obs}

\subsection{A consequence for a single normed space}\label{xsingle}

We recall different types of asymptotic expansions for a normed space from \cite{FHJ}.

\begin{defn}\label{uexp2}
Let $(v_n)_{n=1}^\infty$ be a sequence in a normed space $(Z,\|\cdot\|_Z)$ over $\mathbb C$ or $\mathbb R$. 
Define $(Z_k,\|\cdot\|_{Z_k})=(Z,\|\cdot\|_Z)$ for any integer $k\ge 0$, and 
\begin{equation}\label{sameZ}
    \mathcal Z=((Z_k,\|\cdot\|_{Z_k}))_{k=0}^\infty =((Z,\|\cdot\|))_{k=0}^\infty.
\end{equation}

     We say the sequence $(v_n)_{n=1}^\infty$ has a \emph{strict unitary expansion},
    respectively,
    \emph{unitary expansion, degenerate expansion}, 
    in $Z$ if it has a strict expansion in the sense of Definition \ref{uexp},
    respectively,
    unitary expansion, degenerate expansion,  in the sense of Definition \ref{refinex}, in $\mathcal Z$ defined by \eqref{sameZ}.
\end{defn}

Definition \ref{uexp2} of the strict unitary expansion in $Z$ agrees with that in \cite[Definition 4.1]{FHJ}. Indeed, from \eqref{b3} and \eqref{b2} with $Z_{k-1}=Z_k=Z$ we have $\|w_k\|_Z=1$ which is required in \cite[Definition 4.1]{FHJ}.
Similarly, the definitions of unitary expansions in $Z$ given in Definition \ref{uexp2} above  are the same as those in   \cite[Definition 6.1]{FHJ}.

\begin{cor}\label{maincor}
Let $\mathcal Z=((Z_k,\|\cdot\|_{Z_k}))_{k=0}^\infty$  be as in Theorem \ref{mainlem}, and $(Z_\infty,\|\cdot\|_{Z_\infty})$ be a normed space over $\mathbb K$ such that
 all $Z_k\subset Z_\infty$ with continuous embeddings. 
Then any convergent sequence in $Z_0$ has a subsequence that possesses either a unitary expansion or a degenerate expansion in both 
\begin{enumerate}[label=\rnum]
    \item\label{caseq}  a subsequence $\tilde{\mathcal Z}$  of $\mathcal Z$, and
    \item\label{casingle} space $Z_\infty$.
\end{enumerate}
\end{cor}

The proof of Corollary \ref{maincor} is presented in the Appendix.

\begin{eg}\label{denexam} 
Recall the orthonormal basis $(\varphi_n)_{n=1}^\infty$ of $H$.
For $n\in\mathbb N$, let 
$$v_n=\frac{1}{e^{n^2}}\sum_{k=1}^\infty e^{-kn}\varphi_k.$$    
Then $v_n\in H$ with 
\begin{align*}
|v_n|
&=\frac{1}{e^{n^2+n}}\left|\sum_{k=1}^\infty e^{-(k-1)n}\varphi_k\right|=\frac{1}{e^{n^2+n}(1-e^{-2n})^{1/2}}\to 0\text{ as } n\to\infty.
\end{align*}

\begin{enumerate}[label=\rnum]
    \item Let $v=0$,  $w_k=\varphi_k$, $\Gamma_{k,n}=e^{-kn-n^2}$ and, similar to \eqref{wnkfin}, 
    $$w_n^{(k)}=w_k+\sum_{j=k+1}^\infty \frac{\Gamma_{j,n}}{\Gamma_{k,n}}\varphi_k
    =w_k+\sum_{j=k+1}^\infty e^{-(j-k)n}\varphi_k$$ for $k,n\in\mathbb N$.
One can verify that $w_n^{(k)}\in H$ and 
$$|w_n^{(k)}-w_k|=\frac{e^{-n}}{(1-e^{-2n})^{1/2}}\to 0\text{ as }n\to\infty.$$ 
Thus one obtains the unitary expansion
\begin{equation}\label{strang1}
    v_n\approx 0+\sum_{k=1}^\infty \Gamma_{k,n}w_k \text{ in }H.
\end{equation}

    \item Let $v=0$,  $w_k=0$, $\Gamma_{k,n}=e^{-kn}$ and $w_n^{(k)}=v_n e^{nk}$ for $k,n\in\mathbb N$.
Let $k\in\mathbb N$ be given. We have
$$|w_n^{(k)}|=|v_n| e^{nk}=\frac{e^{nk}}{e^{n^2+n}(1-e^{-2n})^{1/2}} \to 0 \text{ as } n\to\infty.$$
Thus, $w_n^{(k)}\to 0=w_k$ in $H$.
Moreover,
$$ v+\sum_{j=1}^{k-1}\Gamma_{j,n}w_j +\Gamma_{k,n} w_n^{(k)}=\Gamma_{k,n} w_n^{(k)}=v_n.$$
Therefore, we have the degenerate expansion
\begin{equation}\label{strang2}
v_n\approx 0+\sum_{k=1}^\infty \Gamma_{k,n}\cdot 0 \text{ in }H.
\end{equation}
\end{enumerate}

Both expansions \eqref{strang1} and \eqref{strang2} can occur simultaneously because we did not assume that they were obtained from a relaxed expansion via the restructuring process in Theorem \ref{refinethm}.
Note that $v_n$ has the degenerate expansion \eqref{strang2} with all zero terms, but it does not have a trivial unitary expansion $v_n=0$.
\end{eg}

\section{Steady states of the Navier--Stokes equations} \label{results}

Consider a family of the steady state equations corresponding to  \eqref{sc} with $G=\alpha_n>0$ for $n\in\mathbb N$, that is,
\begin{equation}\label{steady}
Av_n+ \alpha_n B(v_n,v_n)=g.
\end{equation}
For the sake of generality, we drop the norm requirement \eqref{g1} and let $g$ be  a fixed function in $H\setminus\{0\}$. Note in the case $|g|\ne 1$ that the actual Grashof number for equation \eqref{steady} is $\alpha_n|g|$.

Assume that $\alpha_n\to\infty$ as $n\to\infty$.
For $n\in \mathbb N$, let $v_n\in V$ be a weak solution of  \eqref{steady}.
It is well-known that $v_n\in D(A)$ for all $n$.

\subsection{The 2D periodic case} \label{2Dappln}
Consider the 2D periodic case. 
Since $v_n\in D(A)$, taking the inner product in $H$ of \eqref{steady} with $Av_n$ and using \eqref{BAzero}, one obtains \begin{equation}\label{DAbound}
    |Av_n|\le |g|.
\end{equation}
Hence, $(v_n)_{n=1}^\infty$ is a bounded sequence in $D(A)$.
By the compact embedding of $D(A)$ into $V$, there exists $v\in V$ and a subsequence of $(v_n)_{n=1}^\infty$, still denoted by $(v_n)_{n=1}^\infty$, such that 
\begin{equation}\label{vlim2}
    v_n\to v\text{ in }V.
\end{equation}

Let $v$ and $(v_n)_{n=1}^\infty$ be any such pair for the rest of this subsection.

\begin{thm}\label{2Dxp}
One has the following.
\begin{enumerate}[label=\tnum]
    \item Given any strictly decreasing sequence $(s_k)_{k=0}^\infty$ of numbers in the interval $(1/2,1)$, set $Z_k=D(A^{s_k})$ and $\mathcal Z=(Z_k)_{k=0}^\infty$.
Then there is a subsequence of $(v_n)_{n=1}^\infty$ that has a strict expansion in $\mathcal Z$.    

\item Let $s$ be any number in the interval $[1/2,1)$.
Then there is a subsequence of $(v_n)_{n=1}^\infty$ that has a unitary expansion or degenerate expansion in $D(A^s)$.    
\end{enumerate}
\end{thm}
\begin{proof}
(i) Let $Z_*=D(A)$. By \eqref{DAbound} and the compact embedding of $Z_*$ into $Z_0$, there is a subsequence of $(v_n)_{n=1}^\infty$, still denoted by $(v_n)_{n=1}^\infty$, such that 
\begin{equation}\label{vlim1}
    v_n\to v\text{ in }Z_0.
\end{equation}
(Note that the limits in \eqref{vlim1} and  \eqref{vlim2} are the same $v$ thanks to the continuous embedding $V\subset Z_0$.)
Then we apply Theorem \ref{mainlem}.

(ii) Choose $(s_k)_{k=0}^\infty$ as in part (i) with $s_k>s$ for all $k\ge 0$. After having a subsequence of $(v_n)_{n=1}^\infty$, still denoted by $(v_n)_{n=1}^\infty$, with $v_n\to v$ in $Z_0$, we apply Corollary \ref{maincor}\ref{casingle} to the normed space $Z_\infty=D(A^s)$.
\end{proof}

Selecting $s=1/2$ and applying Theorem \ref{2Dxp}(ii), we have a unitary or degenerate expansion
\begin{equation}\label{Vx}
v_n\approx v+\sum_{k\in\mathcal N} \Gamma_{k,n}w_k \text{ in } V.    
\end{equation}
\hilite{Next, we find more information about $v$ and $w_k$ in \eqref{Vx}.

\medskip\noindent \textit{Heuristic arguments.} }
We formally write \eqref{Vx} as
\begin{equation}\label{infser}
v_n=v+\Gamma_{1,n}w_1+\Gamma_{2,n}w_2+\cdots+
\Gamma_{k,n}w_k+\cdots,
\end{equation}
where the sum can be without $w_k$, or finite  with $1\le k\le K$, or infinite with $k\in \mathbb N$.
Substituting \eqref{infser} into equation \eqref{steady}, one formally obtains 
\begin{equation}\label{sssub}
\begin{aligned}
&\alpha_n B(v,v)+ (Av-g) +\Gamma_{1,n}Aw_1+ \cdots +
\Gamma_{k,n}Aw_k + \cdots \\
&+\alpha_n\Gamma_{1,n}\Bs(v,w_1) + \cdots +
\alpha_n\Gamma_{k,n}\Bs(v,w_k) + \cdots \\
&+\alpha_n\Gamma_{1,n}\Gamma_{1,n}B(w_1,w_1)
+ \cdots + \sum_{j=1}^k
  \alpha_n\Gamma_{j,n}\Gamma_{k-j,n}
B(w_j,w_{k-j}) + \cdots = 0  .
\end{aligned}
\end{equation}
The sequences of coefficients in \eqref{sssub} are
\begin{equation}\label{coeffdef}
\begin{aligned}
& (\alpha_{n})_{n=1}^{\infty},\quad (1)_{n=1}^{\infty},
\quad (\Gamma_{k,n})_{n=1}^{\infty},\quad 
 (\alpha_{n}\Gamma_{k,n})_{n=1}^{\infty} 
 \text{ for } k\in \mathcal N,\\
&  (\alpha_{n}\Gamma_{j,n}\Gamma_{k,n})_{n=1}^{\infty}
 \text{ for } k,j\in \mathcal N, \ k \ge j.
\end{aligned}
\end{equation}
The idea is to compare the sequences in \eqref{coeffdef} and collect the equivalent ones in equation \eqref{sssub} to write down equations that relate $v$ and $w_k$. In some cases, this idea can be rigorously justified such as in Theorem \ref{thm1} below. In other cases, it can only be applied after some modifications  such as in Assumption \ref{chiassum} and Theorem \ref{v0eg}\ref{cas2} below.

\medskip
We return to our rigorous treatment now.
Following \cite{FHJ} we compare the sequences in \eqref{coeffdef} as $n\to\infty$ in the following way.

\begin{defn} \label{deford} 
Let $\mathcal X$ be the collection of all sequences of positive numbers. 
Given two sequences $\xi=(\xi_n)_{n=1}^{\infty}$, 
$\eta=(\eta_n)_{n=1}^{\infty}$ in $\mathcal X$, we
write  $\xi \succ\eta$, if $\xi_n/\eta_n \to \infty$ as $n \to
\infty$, and  $\xi \sim \eta$ if  $\xi_n/\eta_n \to \lambda$, for some
$\lambda \in (0,\infty)$.  We write $\xi\succsim \eta$ if either  $\xi\succ\eta$ or $\xi \sim\eta$. 

A subset $X$ of $\mathcal X$ is called \emph{totally comparable} if it holds for any $\xi,\eta\in X$ that $\xi\sim \eta$ or $\xi\succ \eta$ or $\eta\succ \xi$.
\end{defn}
Clearly, the relation $\sim $ in Definition \ref{deford} is an equivalence relation on $\mathcal X$ but $\succsim$ is not an order on $\mathcal X$.
For convenience, we use the short-hand notation.
$$\xi_n \sim \eta_n,\text{ respectively, } \xi_n  \succ \eta_n,$$ for sequences $\xi=(\xi_n)_{n=1}^\infty$ and $\eta=(\eta_n)_{n=1}^\infty$ in $\mathcal X$, to mean 
$\xi\sim \eta$, respectively, 
$\xi\succ \eta$.

Let $(\varphi(n))_{n=1}^\infty$ be a subsequence of $(n)_{n=1}^\infty$.
For $X\subset \mathcal X$, define
$$X_\varphi=\left \{ (\xi_{\varphi(n)})_{n=1}^\infty \text{ with } (\xi_n)_{n=1}^\infty \in X\right \}.$$ 
We call $X_\varphi$ a \emph{subsequential set} of $X$.

Note that if $X$ is totally comparable, then so is the subsequential set $X_\varphi$.

Let $\mathcal S$ denote the set of sequences in \eqref{coeffdef}. 
For relations among the sequences in $\mathcal S$, we reproduce Table (3.6) in \cite{FHJ}.

\begin{equation}\label{seqarray}
\begin{array}{cccccccccc}
1    &\succ &\Gamma_{1,n} &\succ &\Gamma_{2,n} &\succ \cdots  &\succ
 &\Gamma_{k,n} &\succ \cdots  \\
\succup & & \succup & & \succup & & & \succup\\
\alpha_n&\succ &\alpha_n\Gamma_{1,n} &\succ &\alpha_n\Gamma_{2,n} &\succ \cdots  &\succ
&\alpha_n\Gamma_{k,n} &\succ \cdots \\
& & \succdn & & \succdn & & & \succdn \\
            & &  \alpha_n\Gamma_{1,n}\Gamma_{1,n} &\succ &\alpha_n\Gamma_{1,n}\Gamma_{2,n} &\succ \cdots &\succ
&\alpha_n\Gamma_{1,n}\Gamma_{k,n} &\succ \cdots \\
& &  & & \succdn & & & \succdn \\
 &  & & &\alpha_n\Gamma_{2,n}\Gamma_{2,n} &\succ \cdots  &\succ &\alpha_n\Gamma_{2,n}\Gamma_{k,n} &\succ \cdots \\
& &  & &  & & & \succdn \\
& & & & & \ddots   & & \vdots \\
& &  & &  & & & \succdn \\
 &  & & & &  & & \alpha_n\Gamma_{k,n}\Gamma_{k,n} &\succ \cdots 
\end{array}
\end{equation}

The relations in \eqref{seqarray} do not guarantee that $\mathcal S$ is totally comparable.
However, it follows the fact $\mathcal S$ is countable and 
 \cite[Lemma 3.1]{FHJ} that  $\mathcal S$ has a subsequential set that is totally comparable.
Together with property \eqref{vsubsum}, we assume the set $\mathcal S$ itself  is totally comparable. (More properties of the set $\mathcal S$ and its associated ordinal numbers were studied in \cite{FHJ}.)

\begin{thm}\label{thm1} 
\hilite{Assume \eqref{steady} and \eqref{Vx}.}  We have the following.
   \begin{enumerate}[label=\tnum]
       \item\label{p0}  The function $v$ satisfies
        \begin{equation}\label{Bzero}
        B(v,v)=0.
        \end{equation}

       \item\label{p1} If \eqref{Vx} is a trivial unitary expansion,  then 
       \begin{equation}\label{Avg}
           Av=g.
       \end{equation}
       
       \item\label{p2} If \eqref{Vx} is not a trivial unitary expansion, then we have the following cases.
   \begin{enumerate}[label=\rnum]
       \item If $1\succsim \alpha_n\Gamma_{1,n}$, then 
       \begin{equation*}
           Av+\mu\Bs(v,w_1)=g, \text{ where }\mu=\lim_{n\to\infty}\alpha_n\Gamma_{1,n}\in[0,\infty).
       \end{equation*}

       \item If $\alpha_n\Gamma_{1,n}\succ 1$, then 
       \begin{equation*}
           \Bs(v,w_1)=0.
       \end{equation*}
    \end{enumerate}
 \end{enumerate}
\end{thm}
\begin{proof} 
(i) Dividing \eqref{steady} by $\alpha_n$, taking the limit in $V'$ as $n\to\infty$ with the use  \hilite{of} the facts $v_n\to v$ in $V$, 
the operator $A$ is continuous from $V$ to $V'$, and the bilinear form $B$ is continuous from $V\times V$ to $V'$, we obtain \eqref{Bzero}.

(ii)  With $v_n=v$ and \eqref{Bzero}, we immediately have \eqref{Avg} from \eqref{steady}.

(iii) In both cases of nontrivial unitary expansion and degenerate expansion, we substitute the expression $v_n=v+\Gamma_{1,n}w_n^{(1)}$ into the steady state equation \eqref{steady} and use property \eqref{Bzero} to obtain
    \begin{equation}\label{se1}
 Av-g + \alpha_n \Gamma_{1,n} \Bs(v,w_n^{(1)}) 
    + \alpha_n\Gamma_{1,n}^2B(w_n^{(1)}, w_n^{(1)})=0  .       
    \end{equation}

For (a) we can simply take the limit in $V'$ as $n\to \infty$ in \eqref{se1}, to obtain 
$$ \Bs(v,w_n^{(1)})\to \Bs(v,w_1),\quad B(w_n^{(1)}, w_n^{(1)})\to B(w_1, w_1), $$
and use that,  in this case,  $\alpha_n\Gamma_{1,n}^2\to 0$.

For (b) we divide \eqref{se1} by $\alpha_n\Gamma_{1,n}$ and then take the limit in $V'$ as $n\to \infty$. 
\end{proof}

\hilite{
\begin{obs}\label{varygex}
To construct a simple example for a nontrivial case of Theorem \ref{thm1}, we consider a sequence of solutions $v_n$  paired with a corresponding sequence of forces $g_n\to g$. 
Suppose 
\begin{equation}\label{gneq}
    Av_n+\alpha_n B(v_n,v_n)=g_n,
\end{equation}
for $g_n\in H$ and $g_n\to g$ in $H$ as $n\to\infty$.
Then we have the sequence $(v_n)_{n=1}^\infty$ is bounded in $D(A)$, and just as Theorem \ref{2Dxp} there is a subsequence, still denoted by $(v_n)_{n=1}^\infty$, that possesses a unitary expansion or degenerate expansion \eqref{Vx}. Then we have exactly the same statements (i), (ii), (iii) as in Theorem \ref{thm1}.
 The proof is exactly the same as that of Theorem \ref{thm1} with $g_n$ replacing $g$ and the fact $g_n\to g$ in $H$ as $n\to\infty$.  
 More complicated results for Galerkin approximations of \eqref{gneq} were obtained in \cite{FHJ}. 
\end{obs}
}

\hilite{
\begin{eg}\label{case2example} 
Let $\nu=1$, $L=2\pi$ so that $\Omega=[0,2\pi]^2$ and $\k0=1$ and  
consider $\bx=(x,y)\in \mathbb R^2$,
\begin{align*}
    \bu_n(\bx)&=\begin{bmatrix} \sin(y) \\  n\sum_{m=2}^\infty c_m \sin(mx) \end{bmatrix},\\
{\bf F}_n(\bx) &= \begin{bmatrix} \sin(y)\\ n\sum_{m=2}^\infty  m^2 c_m \sin(mx) \end{bmatrix}
+ n\sum_{m=2}^\infty c_m  \begin{bmatrix}\sin(mx)\cos(y) \\ m\sin(y)\cos(mx) \end{bmatrix},
\quad \bF_n=\mathcal P{\bf F}_n,
\end{align*}
where  $c_m\in\mathbb R$, for $m\ge 2$, not all zero such that $\sum_{m=2}^\infty m^4 c_m^2<\infty$. 
Then $\bu_n\in D(A)$ and $\bF_n\in H$. We can verify that
$$-\Delta \bu_n+(\bu_n\cdot\nabla) \bu_n={\bf F}_n.$$
By applying the Helmholtz-Leray projection $\mathcal P$ to this equation, we obtain
$$A\bu_n+B(\bu_n,\bu_n)=\bF_n.$$
We find
\begin{align*}
    & {\bf F}_n(\bx) = \begin{bmatrix} \sin(y)\\ n\sum_{m=2}^\infty  m^2 c_m \sin(mx) \end{bmatrix}
+ \frac{n}{4\im}\sum_{m=2}^\infty c_m \begin{bmatrix} (\ex^{m\im x}-\ex^{-m\im x})(\ex^{\im y}+\ex^{-\im y}) \\ m(\ex^{\im y}-\ex^{-\im y})(\ex^{m\im x}+\ex^{-m\im x}) \end{bmatrix}\\
&= \begin{bmatrix} \sin(y)\\ n\sum_{m=2}^\infty  m^2 c_m \sin(mx) \end{bmatrix}\\
&\quad -\frac{\im n} {4}\sum_{m=2}^\infty c_m
\left\{\ex^{\im(m,1)\cdot \bx}\begin{bmatrix} 1 \\m \end{bmatrix}
      -\ex^{\im(-m,-1)\cdot \bx}\begin{bmatrix} 1 \\m \end{bmatrix}
      +\ex^{\im(m,-1)\cdot \bx}\begin{bmatrix} 1 \\ -m \end{bmatrix}
      -\ex^{\im(-m,1)\cdot \bx}\begin{bmatrix} 1 \\ -m \end{bmatrix} \right\}.
\end{align*}
Note that the first function $\begin{bmatrix} \sin(y)\\ n\sum_{m=2}^\infty  m^2 c_m \sin(mx) \end{bmatrix}$ of ${\bf F}_n$ is already divergence-free. We calculate the projection
\begin{align*}
\bF_n(\bx) &=\begin{bmatrix} \sin(y)\\ n\sum_{m=2}^\infty  m^2 c_m \sin(mx) \end{bmatrix}\\
&\quad -\frac{\im n} {4}\sum_{m=2}^\infty c_m\frac{m^2-1}{m^2+1}
\left\{(\ex^{\im(m,1)\cdot \bx}-\ex^{\im(-m,-1)\cdot \bx})\begin{bmatrix} -1 \\ m\end{bmatrix}
      -(\ex^{\im(m,-1)\cdot \bx}-\ex^{\im(-m,1)\cdot \bx})\begin{bmatrix} 1 \\ m \end{bmatrix} \right\} .
      \end{align*}
Thus, computing the $L^2$-norm of $\bF_n$ with the use of Parseval's identity, we find
$$G=\alpha_n=\|\bF_n\|_{L^2(\Omega)^2}=\sqrt 2\pi\sqrt{1+c_*^2 n^2} = \sqrt 2 \pi c_*\sqrt{n^2+1/c_*^2},$$ 
where 
$$c_*=\sqrt{\sum_{m=2}^\infty m^4c_m^2+\frac12\sum_{m=2}^\infty c_m^2\frac{(m^2-1)^2}{m^2+1}}>0.$$

After the change of variables \eqref{changevar}, particularly, $\tilde \bx=\bx$, $\tilde \Omega=\Omega$, ${\bf v}_n(\bx)=\bu_n(\bx)/\alpha_n$,  we write ${\bf{v_n}}=v_n$ as
\begin{equation}\label{vnexeg}
v_n=v+\Gamma_{1,n} w _1+\Gamma_{2,n} w_2, 
\end{equation}
where 
\begin{align*}
v&={\bf v}=\mu_0\begin{bmatrix} 0 \\ \sum_{m=2}^\infty   c_m \sin(mx) \end{bmatrix},\quad 
\Gamma_{1,n}=\frac{\sqrt 2\pi }{\alpha_n},\quad 
w _1=\bfw_1=\frac{1}{\sqrt{2}\pi}\begin{bmatrix} \sin(y) \\ 0 \end{bmatrix},\\
&\quad  
\Gamma_{2,n}=\|\tilde{\bfw}_2\| \left(  \mu_0 - \frac{n}{\alpha_n} \right ),\quad 
w_2=\bfw_2=\tilde{\bfw}_2/\|\tilde{\bfw}_2 \|,\text{ with } \tilde{\bfw}_2   =-\begin{bmatrix} 0\\ \sum_{m=2}^\infty   c_m \sin(mx) \end{bmatrix}
\end{align*}
and $\mu_0=\lim_{n\to\infty}n/\alpha_n=1/(\sqrt 2 \pi c_*)$. Note that 
$\|\tilde{\bfw}_2\|=\sqrt 2\pi \sqrt{\sum_{m=2}^\infty m^2c_m^2}>0,$
$w_1$ and $w_2$ belong to $V$ with norms $1$.
Clearly, $\Gamma_{1,n}\to 0$ and $\Gamma_{2,n}>0$,
$$\Gamma_{2,n}=\frac{\mu_0^2\alpha_n^2-n^2}{\alpha_n(\mu_0\alpha_n+n)}\|\tilde{\bfw}_2\|=\frac{1}{c_*^2\alpha_n(\mu_0\alpha_n+n)}\|\tilde{\bfw}_2\|,$$
hence $\Gamma_{2,n}/\Gamma_{1,n}\to 0$.
Then \eqref{vnexeg} is a unitary expansion of $v_n$ in $V$.

Let ${\bf g}_n(\bx)=\bF_n(\bx)/\alpha_n$, and $g_n={\bf g}_n(\cdot)$. Then we obtain equation \eqref{gneq} in Remark \ref{varygex}.
Note that 
$$\frac{{\bf F}_n(\bx)}{\alpha_n}=\frac{1}{\alpha_n}\begin{bmatrix} \sin(y)+ n\sum_{m=2}^\infty \sin(mx)\cos(y) \\ 
n\sum_{m=2}^\infty (m^2 \sin(mx) +m \sin(y)\cos(mx)) \end{bmatrix}$$
which converges in $H$ to 
$$ {\bf{g}}(\bx)= \mu_0 \begin{bmatrix}  \sum_{m=2}^\infty \sin(mx)\cos(y)  \\  \sum_{m=2}^\infty (m^2 \sin(mx) +m \sin(y)\cos(mx)) \end{bmatrix}.$$
Hence, $g_n=\mathcal P {\bf F}_n/\alpha_n \to g:=\mathcal P{\bf g}$ in $H$.
We have $\mu=\lim_{n \to \infty}\alpha_n \Gamma_{1,n}=\sqrt 2\pi$ and it is easy to confirm that
$$
-\Delta \bfv  + \mu ((\bfv\cdot \nabla) \bfw_1 + (\bfw_1\cdot \nabla) \bfv)=\bfg 
$$
and thus, by applying the projection $\mathcal P$ to this equation,  (iii)(a) of Theorem \ref{thm1} holds, see Remark \ref{varygex}.
\end{eg}
}

\begin{figure}[ht]
\psfrag{v=0}{\tiny$v=0$}
\psfrag{1ge1}{\tiny$\alpha_n\Gamma_{1,n} \succsim 1$}
\psfrag{Bw1=0}{\tiny$B(w_1,w_1)=0$}
\psfrag{1gt1}{\tiny$\alpha_n\Gamma_{1,n}^2\succ 1$}
\psfrag{12ge1}{\tiny$\alpha_n\Gamma_{1,n}\Gamma_{2,n} \succsim 1$}
\psfrag{12gt1}{\tiny$\alpha_n\Gamma_{1,n}\Gamma_{2,n} \succ 1$}
\psfrag{12=1}{\tiny$\alpha_n\Gamma_{1,n}\Gamma_{2,n} \sim 1$}
\psfrag{Bs12=0}{\tiny$B_s(w_1,w_2)=0$}
\psfrag{muBs12=g}{\tiny$\mu B_s(w_1,w_2)=g$}
\psfrag{muBs11=g}{\tiny$\mu_* B_s(w_1,w_1)=g$}

\psfrag{chi0orge}{\tiny$\chi_n=0$ $\forall \ n$, or $\Gamma_{1,n}\succsim |\chi_n|$}
\psfrag{approx}{\tiny$ v_n\approx 0+\sum_{k=1}^\infty \Gamma_{k,n}\cdot 0 \text{ in } V$}
\psfrag{chi0orgt}{\tiny$\chi_n=0$ $\forall \ n$, or $\Gamma_{1,n}\succ |\chi_n|$}
\psfrag{chi0}{\tiny$\chi_n=0$ $\forall \ n$, or}
\psfrag{or12ge}{\tiny$\alpha_n\Gamma_{1,n}\Gamma_{2,n}\succsim |\chi_n|$}
\psfrag{or1ge}{\tiny$\alpha_n\Gamma_{1,n}\succsim |\chi_n|$}
\psfrag{mu1mu2}{\tiny$ \mu_1Aw_1 +\mu_2\Bs(w_1,w_2)=g$}
\psfrag{mustar12}{\tiny$\mu_* \mu_1\|w_1\|^2 = \mu_2\langle g,w_2\rangle$}
\psfrag{for12}{\tiny $\exists\  \mu_1, \mu_2\in\mathbb R \text{ with } \mu_1\mu_2>0$}
\psfrag{for2}{\tiny for some $\mu_2 > 0$}
\psfrag{mu2}{\tiny$ \mu_1Aw_1 +\mu_2\Bs(w_1,w_2)=0$}
\psfrag{mustar2}{\tiny$\mu_*\|w_1\|^2 = \mu_2\langle g,w_2\rangle$,}
\psfrag{mu2B12=g}{\tiny$\mu_2 B_s(w_1,w_2)=g$ for some $\mu_2 \neq 0$}
\psfrag{gw2}{\tiny$\langle g,w_2 \rangle =\langle B(w_2,w_2),w_1 \rangle_{V',V}  = 0$}
\psfrag{chi0or12gt}{\tiny $\chi_n=0 \ \forall \ n$, or $\alpha_n \Gamma_{1,n}\Gamma_{2,n} \succ |\chi_n|$}
 \psfrag{chi0or1gt}{\tiny $\chi_n=0 \ \forall \ n$, or $\alpha_n \Gamma_{1,n} \succ |\chi_n|$}
\psfrag{12=chi}{\tiny$\alpha_n\Gamma_{1,n}\Gamma_{2,n}\sim |\chi_n|$}
\psfrag{1=chi}{\tiny$\alpha_n\Gamma_{1,n}\sim |\chi_n|$}
\psfrag{1=1}{\tiny$\alpha_n\Gamma_{1,n}^2\sim 1$}
\psfrag{1gt2}{\tiny$1 \succ \alpha_n\Gamma_{2,n}$}
\psfrag{2=1}{\tiny$1 \sim\alpha_n\Gamma_{2,n}$}
\psfrag{2gt1}{\tiny$\alpha_n\Gamma_{2,n}\succ 1$}
\psfrag{(i)}{\tiny(i)}
\psfrag{(ii)}{\tiny(ii)}
\psfrag{(1)}{\tiny(1)}
\psfrag{(2)}{\tiny(2)}
\psfrag{(3)}{\tiny(3)}
\psfrag{(a)}{\tiny(a)}
\psfrag{(b)}{\tiny(b)}
\psfrag{*}{\hspace*{-.1em}*}
\psfrag{Assumption}{\tiny * Assumption \ref{chiassum}}
\psfrag{w2exists}{\tiny $w_2 \ \exists$ in \eqref{Vx}}
\centerline{\includegraphics[scale=.55]{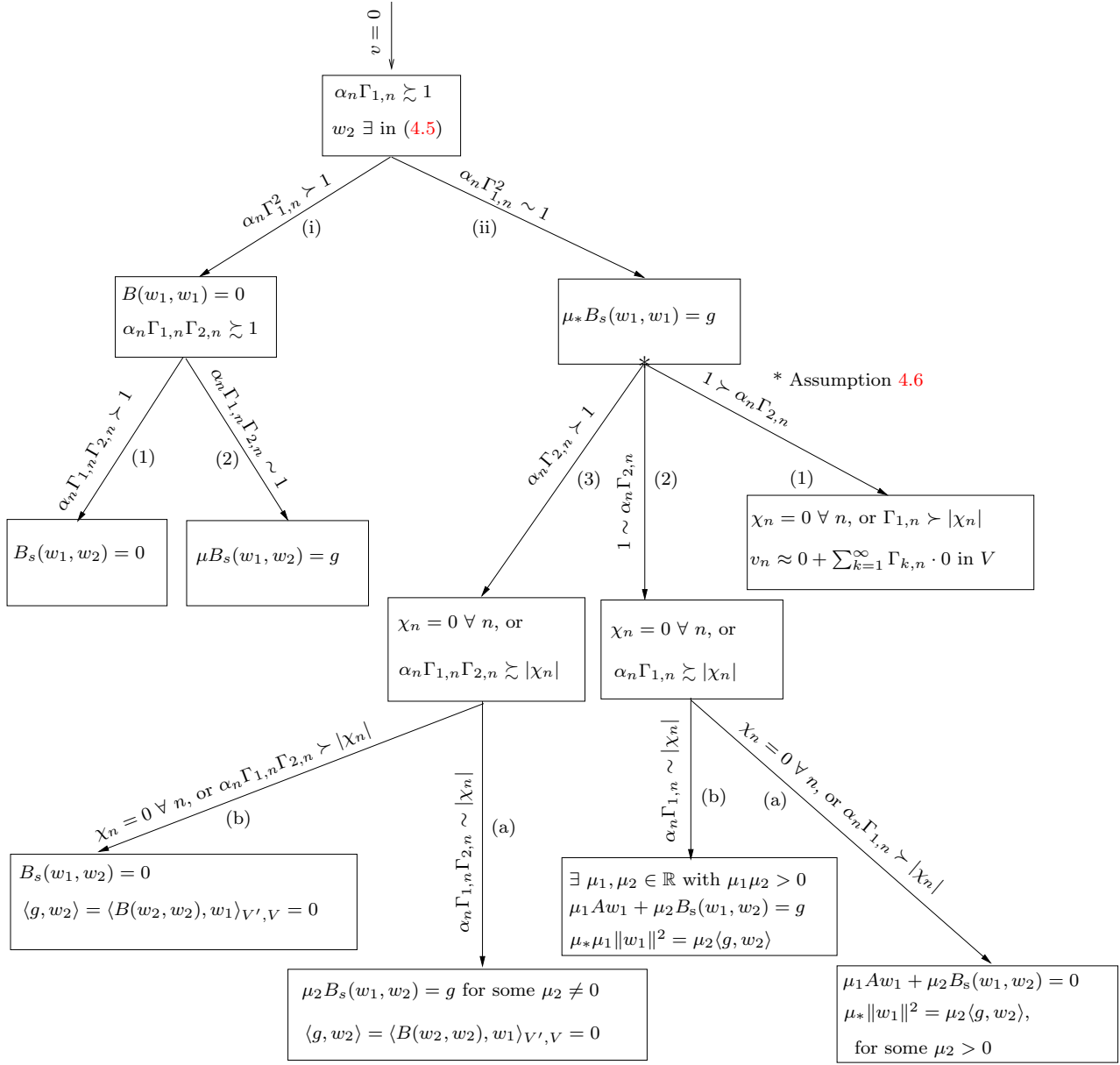}}
\caption{Implications if $v=0$, as stated in Theorem \ref{v0eg}}
\label{flowchart}
\end{figure}

Theorem \ref{thm1} can be explicated in more details for special cases such as $v=0$ and $v=A^{-1}g$ in Theorems \ref{v0eg} and \ref{vAinvg} below. In the following, we specify where we use Assumption \ref{chiassum}.

\begin{assum}\label{chiassum}
Consider the case $\alpha_n \Gamma_{1,n}^2\sim 1$. Define 
\begin{equation}\label{xn}
    \lim_{n\to\infty} \alpha_n \Gamma_{1,n}^2= \mu_*\in(0,\infty)\text{ and set }\chi_n=1-\alpha_n\Gamma_{1,n}^2/\mu_*.
\end{equation}

By the virtue of \cite[Lemma A.4]{FHJ} and using a subsequential set, we further assume that the sequence $(\chi_n)_{n=1}^\infty$ satisfies
either

{\rm (S1)} $\chi_n=0$ for all $n$, or

{\rm (S2)} $\chi_n>0$ for all $n$ and $\mathcal S\cup \{(\chi_n)_{n=1}^\infty\}$ is totally comparable, or

{\rm (S3)} $\chi_n<0$ for all $n$ and $\mathcal S\cup \{(-\chi_n)_{n=1}^\infty\}$ is totally comparable. 
\end{assum}

\hilite{A flowchart description of the next result is given in Figure \ref{flowchart}.}

\begin{thm} \label{v0eg} Assume $v=0$. 
Then $\alpha_n\Gamma_{1,n}^2 \succsim 1$ and $w_2$ exists in \eqref{Vx}.
\begin{enumerate}[label=\tnum]
    \item\label{cas1} Case $\alpha_n \Gamma_{1,n}^2 \succ 1$. Then 
    \begin{equation}\label{Bwzero} B(w_1,w_1)=0 
    \end{equation}
    and $\alpha_n \Gamma_{1,n}\Gamma_{2,n} \succsim 1.$
    In addition,
    \begin{enumerate}[label=\nnum]
        \item\label{f1} if $\alpha_n \Gamma_{1,n}\Gamma_{2,n} \succ 1$, then $\Bs(w_1,w_2)=0$, while 
        \item\label{f2} if $\alpha_n \Gamma_{1,n}\Gamma_{2,n} \sim 1$, then 
        \begin{equation} \label{Blg}
        \mu \Bs(w_1,w_2)=g\text{ for some $\mu>0$ }
        \end{equation}
        and $\langle g,w_1\rangle=0$.
    \end{enumerate}
    \smallskip
    \item\label{cas2} Case $\alpha_n \Gamma_{1,n}^2 \sim 1$. Let $\mu_*$ be defined in \eqref{xn}. Then 
    \begin{equation}\label{Bg}
        \mu_* B(w_1,w_1)=g.
    \end{equation}

    \hilite{Under} Assumption \ref{chiassum}, one has more specific cases below.
    \smallskip
    \begin{enumerate}[label=\nnum]
        \item\label{f5} Case $1\succ \alpha_n \Gamma_{2,n}$. Then $\chi_n=0$ for all $n$,  or $\Gamma_{1,n}\succ |\chi_n|$. In this case, we have the degenerate expansion 
        \begin{equation}\label{zerodegen}
            v_n\approx 0+\sum_{k=1}^\infty \Gamma_{k,n}\cdot 0 \text{ in } V.
        \end{equation}
    
        \item\label{f3} Case $\alpha_n \Gamma_{2,n} \sim 1$. Then $\chi_n=0$ for all $n$, or $\Gamma_{1,n} \succsim |\chi_n|$. 
            \smallskip
        \begin{enumerate}[label=\rnum]
            \item\label{last3} If $\Gamma_{1,n} \sim |\chi_n|$, then 
            \begin{equation}\label{BL1}
                \mu_1Aw_1 +\mu_2\Bs(w_1,w_2)=g  \text{ for some } \mu_1, \mu_2\in\mathbb R \text{ with } \mu_1\mu_2>0,
            \end{equation}
            and 
            $\mu_* \mu_1\|w_1\|^2 = \mu_2\langle g,w_2\rangle$,
            
            \item\label{last4} If $\chi_n=0$ for all $n$,  or $\Gamma_{1,n} \succ |\chi_n|$, then 
            \begin{equation}\label{BL2}
                Aw_1 +\mu_2\Bs(w_1,w_2)=0 \text{ for some }  \mu_2 > 0
            \end{equation}
            and $\mu_* \|w_1\|^2 = \mu_2 \langle g,w_2\rangle$.
        \end{enumerate}
    \smallskip
        \item\label{f4} Case $\alpha_n \Gamma_{2,n} \succ 1$.
        Then $\chi_n=0$ for all $n$, or $\alpha_n \Gamma_{1,n}\Gamma_{2,n} \succsim |\chi_n|$. 
        \begin{enumerate} [label=\rnum]   \smallskip        
            \item\label{last1} If $\alpha_n \Gamma_{1,n}\Gamma_{2,n} \sim |\chi_n|$, then 
            \begin{equation}\label{BL3}
            \mu_2\Bs(w_1,w_2)=g\text{ for some $\mu_2 \neq 0$,}
            \end{equation}
           \begin{equation}\label{BL5}
                \langle g,w_2 \rangle =\langle B(w_2,w_2),w_1 \rangle_{V',V}  = 0.
            \end{equation}
           \item\label{last2} If $\chi_n=0$ for all $n$, or $\alpha_n \Gamma_{1,n}\Gamma_{2,n} \succ |\chi_n|$, then \begin{equation}\label{BL4}
           \Bs(w_1,w_2)=0,
           \end{equation}
            and \eqref{BL5} holds.
        \end{enumerate} 
    \end{enumerate}
\end{enumerate}
\end{thm}
\begin{proof} Because $g\ne 0$, the steady state $v_n$ is nonzero, that is, $v_n\ne v$.
Hence, expansion \eqref{infser} is nontrivial. Consequently, $w_1$ exists.
Using $v_n=\Gamma_{1,n}w_n^{(1)}$, we have from \eqref{steady} that
\begin{equation}\label{t1}
    \Gamma_{1,n}Aw_n^{(1)}+\alpha_n \Gamma_{1,n}^2B(w_n^{(1)},w_n^{(1)})=g.
\end{equation}
If $1\succ \alpha_n\Gamma_{1,n}^2$, then passing $n\to\infty$ in \eqref{t1} gives $g=0$, a contradiction. Thus, we must have 
$\alpha_n\Gamma_{1,n}^2\succsim 1$.

\medskip
We now prove $w_2$ exists. It certainly does when \eqref{Vx} is a degenerate expansion. Consider the case \eqref{Vx} is a unitary expansion.
Suppose the unitary expansion \eqref{Vx} stops at $w_1$, that is $v_n=\Gamma_{1,n}w_1$ with $\|w_1\|=1$. Then
\begin{equation}\label{t2}
    \Gamma_{1,n}Aw_1+\alpha_n \Gamma_{1,n}^2B(w_1,w_1)=g.
\end{equation}

\medskip\noindent
\emph{Case A. $\alpha_n\Gamma_{1,n}^2\succ 1$.} Dividing  \eqref{t2}  by $\alpha_n \Gamma_{1,n}^2$ and then passing $n\to\infty$ yields $B(w_1,w_1)=0$.
Using this fact in \eqref{t2} implies $\Gamma_{1,n}Aw_1=g$, which yields $g=0$, a contradiction.

\medskip\noindent
\emph{Case B. $\alpha_n\Gamma_{1,n}^2\sim 1$.} 
By \cite[Lemma A.4]{FHJ} and using a subsequential set of $\mathcal S$, we can 
\hilite{assume, without loss of generality, that Assumption \ref{chiassum} holds now.} 
Passing $n\to\infty$ in \eqref{t2} gives 
\begin{equation}\label{t4}
    \mu_* B(w_1,w_1)=g.
\end{equation} 

We can rewrite \eqref{t2} as
\begin{equation}\label{t3}
        \Gamma_{1,n}Aw_1=\chi_n g.
\end{equation}

\begin{itemize}
    \item 
        If $\chi_n=0$ for all $n$, or  $\Gamma_{1,n}\succ |\chi_n|$, then we infer from \eqref{t3} that $Aw_1=0$, which is a contradiction.
    \item 
        If $|\chi_n|\succ \Gamma_{1,n}$, then $g=0$, a contradiction.
    \item 
        If $\Gamma_{1,n}\sim |\chi_n|$, then \eqref{t3} implies $\mu_1 Aw_1=g$ for some $\mu_1\ne 0$.
Combining this with \eqref{t4} gives
\begin{equation*}
    \mu_1 Aw_1=\mu_* B(w_1,w_1).
\end{equation*}
Taking the product in $V'\times V$ of this equation with $w_1$, and applying \eqref{Auu}, \eqref{BV}, we deduce $w_1=0$, a contradiction.
\end{itemize}

\noindent 
Since both cases A and B yield contradictions, the unitary expansion \eqref{Vx} cannot stop at $w_1$. Thus, $w_2$ exists.

\medskip
Now, we write $v_n=\Gamma_{1,n}w_1+\Gamma_{2,n}w_n^{(2)}$ and substitute it into the steady state equation \eqref{steady}  to obtain
\begin{multline} \label{z2m}
\Gamma_{1,n}Aw_1+\Gamma_{2,n}Aw_n^{(2)} +\alpha_n\Gamma_{1,n}^2 B(w_1,w_1)
+\alpha_n \Gamma_{1,n}\Gamma_{2,n}\Bs(w_1,w_n^{(2)})\\
 +\alpha_n \Gamma_{2,n}^2B(w_n^{(2)},w_n^{(2)})=g.
\end{multline}
\begin{enumerate}[label=\tnum]
    \item \textit{Case $\alpha_n\Gamma_{1,n}^2 \succ 1$.} Since $\alpha_n\Gamma_{1,n}^2\to\infty$ we have also $\alpha_n\Gamma_{1,n}\to\infty$. Dividing \eqref{t1} by $\alpha_n\Gamma_{1,n}^2$ and taking $n\to\infty$, we obtain \eqref{Bwzero}, so that \eqref{z2m} reduces to 
\begin{equation} \label{z4m}
\Gamma_{1,n}Aw_1+\Gamma_{2,n}Aw_n^{(2)} +\alpha_n \Gamma_{1,n}\Gamma_{2,n}\Bs(w_1,w_n^{(2)})
 +\alpha_n \Gamma_{2,n}^2B(w_n^{(2)},w_n^{(2)})=g.
\end{equation}

If $1 \succ \alpha_n \Gamma_{1,n}\Gamma_{2,n}$, then all terms on the left in \eqref{z4m} tend to 0 as $n\to \infty$, which would mean $g=0$, a contradiction. Hence, $\alpha_n \Gamma_{1,n}\Gamma_{2,n}\succsim  1$, and consequently, 
$\alpha_n \Gamma_{2,n}\to\infty$.

Then part \ref{f1} and the first identity \eqref{Blg} of part \ref{f2} follow from dividing the equation \eqref{z4m} by $\alpha_n \Gamma_{1,n}\Gamma_{2,n}$, passing $n\to\infty$.
For the second identity of part \ref{f2}, by taking the product in $V'\times V$ of \eqref{Blg} with $w_1$, we obtain
$$\langle g,w_1\rangle = \mu\langle B(w_1,w_2),w_1\rangle_{V',V}+\mu\langle B(w_2,w_1),w_1\rangle_{V',V}=-\mu\langle B(w_1,w_1),w_2\rangle_{V',V} + 0=0.$$ 
Note that we applied both relations in \eqref{BV} and used \eqref{Bwzero}. 

\item \textit{Case $\alpha_n \Gamma_{1,n}^2 \sim 1$.} We take $n \to \infty$ in \eqref{t1} and use the limit in \eqref{xn} to obtain  \eqref{Bg}.
As a consequence of \eqref{Bg}, we have
\begin{equation}\label{gw1}
\langle g,w_1\rangle=0.
\end{equation}
Using this identity \eqref{Bg}, we rewrite \eqref{z2m} as
\begin{align} \label{z3m}
\Gamma_{1,n}Aw_1+\Gamma_{2,n}Aw_n^{(2)}
+\alpha_n \Gamma_{1,n}\Gamma_{2,n}\Bs(w_1,w_n^{(2)})
+\alpha_n \Gamma_{2,n}^2B(w_n^{(2)},w_n^{(2)})=\chi_n g.
\end{align}

\hilite{Now, under Assumption \ref{chiassum}, we have the following.}

\begin{enumerate}[label=\nnum]

\item Case $1\succ \alpha_n \Gamma_{2,n}$. Then $\Gamma_{1,n}\succ \alpha_n \Gamma_{1,n}\Gamma_{2,n}\succ \alpha_n \Gamma_{2,n}^2$.  We next consider the three possibilities that follow.
\begin{itemize}
\item   
        If $\chi_n=0$ for all $n$,  or $\Gamma_{1,n}\succ |\chi_n|$, then \eqref{z3m} implies  $Aw_1=0$. Hence $w_1=0$, and, thanks to \eqref{Vx} being an either unitary or degenerate expansion, \eqref{Vx}  must be the degenerate expansion \eqref{zerodegen}.
 
\item
        If $|\chi_n|\succ \Gamma_{1,n}$, then \eqref{z3m} implies $g=0$ which is a contradiction.

\item 
        If $|\chi_n|\sim \Gamma_{1,n}$, then \eqref{z3m} implies 
$g=\mu_1 Aw_1$, with $\mu_1\ne 0.$ Taking the product in $V'\times V$ of this equation with $w_1$ and using \eqref{gw1}, we deduce $w_1=0$, a contradiction.
\end{itemize}

    \item \emph{Case} $ \alpha_n \Gamma_{2,n}\sim 1$. Then $\Gamma_{1,n}\sim \alpha_n\Gamma_{1,n}\Gamma_{2,n}\succ \alpha_n \Gamma_{2,n}^2$.  If $|\chi_n|\succ \Gamma_{1,n}$, then \eqref{z3m} implies $g=0$, a contradiction. 
    Thus, $\chi_n=0$ or $\Gamma_{1,n}\succsim |\chi_n|$.  
\begin{enumerate}[label=\rnum]
        \item If $|\chi_n|\sim \Gamma_{1,n}$, then, since 
 $w_n^{(2)}\to w_2$, equation \eqref{z3m} implies the equation in \eqref{BL1} with 
$$ 
\mu_1=\lim_{n\to\infty}\frac{\Gamma_{1,n}}{\chi_n}\ne 0 \text{ and }
\mu_2=\lim_{n\to\infty}\frac{\alpha_n\Gamma_{2,n}}{\chi_n}\ne 0.$$
Multiplying these two limits yields $\mu_1\mu_2>0$.
Taking the product in $V'\times V$ of \eqref{BL1} with $w_1$ and using property \eqref{gw1} give
$$\mu_1\|w_1\|^2=\mu_2 \langle B(w_1,w_1),w_2\rangle_{V',V} 
=(\mu_2/\mu_*) \langle g,w_2\rangle,$$
which implies the second identity of part (a).

       \item If $\chi_n=0$ for all $n$, or $\Gamma_{1,n}\succ |\chi_n|$, then \eqref{z3m} implies \eqref{BL2}.
By taking the product in $V'\times V$ of  equation \eqref{BL2}  with $w_1$, we obain
$$ \|w_1\|^2 -\mu_2 \langle B(w_1,w_1),w_2\rangle_{V',V} =0 $$ 
and hence, thanks to \eqref{Bg}, $(\mu_2/\mu_*) \langle g,w_2\rangle=\|w_1\|^2$.
Thus, the second identity of part (b) follows.
    \end{enumerate}
    
    \item  \emph{Case} $ \alpha_n \Gamma_{2,n}\succ 1$. Then $\alpha_n \Gamma_{1,n}\Gamma_{2,n}\succ \Gamma_{1,n}$.  If $|\chi_n|\succ \alpha_n \Gamma_{1,n}\Gamma_{2,n}$, then \eqref{z3m} implies  $g=0$ again, which is a contradiction. 

\begin{enumerate}[label=\rnum]
    \item If $|\chi_n|\sim \alpha_n \Gamma_{1,n}\Gamma_{2,n}$. 
Then \eqref{z3m} yields \eqref{BL3}.
Taking the product in $V'\times V$ of equation \eqref{BL3} with $w_1$ and using \eqref{gw1}, \eqref{BV} and \eqref{Bg}, we find
\begin{align*}
0=\langle g,w_1 \rangle 
=\mu_2 \langle B(w_1,w_2),w_1\rangle_{V',V} =-\mu_2 \langle B(w_1,w_1),w_2\rangle_{V',V}
=-\frac{\mu_2}{\mu_*}\langle g,w_2 \rangle.
\end{align*}    
Now the product of equation \eqref{BL3} with $w_2$ yields
\begin{align*}
0=\langle g,w_2 \rangle 
 =\mu_2 \langle B(w_2,w_1),w_2 \rangle_{V',V} 
=-\mu_2 \langle B(w_2,w_2),w_1\rangle_{V',V}.
\end{align*}    

   \item  If $\chi_n=0$ for all $n$, or $\alpha_n \Gamma_{1,n}\Gamma_{2,n} \succ |\chi_n|$, then \eqref{z3m} gives
 \eqref{BL4}.
Then, again, taking the product in $V'\times V$  of equation \eqref{BL4} with $w_1$ yields $\langle g,w_2\rangle=0$, and taking the product in $V'\times V$  of equation \eqref{BL4} with $w_2$ gives $\langle B(w_2,w_2),w_1\rangle_{V',V}=0$.
\end{enumerate}
\end{enumerate}
\end{enumerate}
\end{proof}

According to Theorem \ref{thm1}, if \eqref{Vx} is a trivial unitary expansion, then $Av=g$.  
Below, we study the case when $Av=g$, but \eqref{Vx} is not a trivial unitary expansion.

\begin{thm} \label{vAinvg} 
Assume $v=A^{-1}g$ and  \eqref{Vx} is not a trivial unitary expansion.
Then 
\begin{equation}\label{B1}
    \Bs(v,w_1)=0
\end{equation}
and there exists $w_2$ in \eqref{Vx}. Moreover, below are all scenarios for comparing the sequences $(\Gamma_{1,n})_{n=1}^\infty$, $(\alpha_n \Gamma_{2,n})_{n=1}^\infty$ and $(\alpha_n \Gamma_{1,n}^2)_{n=1}^\infty$.
\begin{enumerate}[label=\tnum]
   \item\label{m1} If $\Gamma_{1,n} \succ \alpha_n \Gamma_{2,n} ,\alpha_n \Gamma_{1,n}^2 $ or $\Gamma_{1,n} \sim \alpha_n \Gamma_{1,n}^2 \succ \alpha_n \Gamma_{2,n}  $, then one has the degenerate expansion
   \begin{equation}\label{degen2}
       v_n\approx v+\sum_{k=1}^\infty \Gamma_{k,n}\cdot 0 \text{ in } V.
   \end{equation}
    \item\label{m2} If $\alpha_n \Gamma_{2,n} \succ \Gamma_{1,n} ,\alpha_n \Gamma_{1,n}^2 $, then
    $\Bs(v,w_2)=0 .$

    \item\label{m3} If $\alpha_n \Gamma_{1,n}^2 \succ \alpha_n \Gamma_{2,n} ,\Gamma_{1,n} $, then
    $B(w_1,w_1)=0 .$

    \item\label{m4} If $\Gamma_{1,n} \sim  \alpha_n \Gamma_{2,n} \succ \alpha_n \Gamma_{1,n}^2 $, then
    $Aw_1+\mu \Bs(v,w_2)=0,$ for some $\mu>0 .$

    \item\label{m5} If $\alpha_n \Gamma_{2,n} \sim  \alpha_n \Gamma_{1,n}^2 \succ \Gamma_{1,n}  $, then
    $\Bs(v,w_2)+\mu B(w_1,w_1)=0$, for some $\mu> 0 .$
    
    \item\label{m6} If $\Gamma_{1,n} \sim \alpha_n \Gamma_{2,n} \sim  \alpha_n \Gamma_{1,n}^2  $, then
    $Aw_1+\mu_1\Bs(v,w_2)+\mu_2 B(w_1,w_1)=0$, for some $\mu_1,\mu_2>0 .$
\end{enumerate}
\end{thm}
\begin{proof} 
Recall from \eqref{Bzero} that $B(v,v)=0$.
 Substituting $v_n=v+\Gamma_{1,n}w_n^{(1)}$ into equation \eqref{steady} gives
\begin{equation}\label{t5}
\Gamma_{1,n}Aw_n^{(1)}
+\alpha_n \Gamma_{1,n} \Bs(v,w_n^{(1)})  +\alpha_n \Gamma_{1,n}^2B(w_n^{(1)},w_n^{(1)})=0,    
\end{equation}
Dividing equation \eqref{t5} by $\alpha_n\Gamma_{1,n}$ and letting $n\to\infty$ gives \eqref{B1}.

In the case  \eqref{Vx} is a degenerate expansion, $w_2$ always exists.
It remains to consider the case \eqref{Vx} is a unitary expansion  and it stops at $w_1$, that is, $v_n=v+\Gamma_{1,n}w_1$, with $w_1\ne 0$.
Then one has equation \eqref{t5} with $w_1$ replacing $w_n^{(1)}$ which, together with \eqref{B1}, implies
\begin{equation}\label{t6}
       Aw_1+ \alpha_n \Gamma_{1,n}B(w_1,w_1)=0.
\end{equation}

\begin{itemize}
    \item If $1\succ \alpha_n \Gamma_{1,n}$, then \eqref{t6} implies $Aw_1=0$ which is a contradiction.

    \item If $\alpha_n \Gamma_{1,n}\succ 1$, then \eqref{t6} implies $B(w_1,w_1)=0$, which, thanks to \eqref{t6} again, yields $Aw_1=0$, a contradiction.

    \item If $1\sim \alpha_n \Gamma_{1,n}$, then \eqref{t6} implies 
$Aw_1+\mu B(w_1,w_1)=0$ for some number $\mu>0$. Taking the product in $V'\times V$ of this equation with $w_1$ yields $w_1=0$ which is a contradiction.
\end{itemize}
All three scenarios above yield a contradiction, hence $w_2$ exists.

Now, we can write $v_n=v+\Gamma_{1,n}w_1+\Gamma_{2,n}w_n^{(2)}$ and substitute this expression into equation \eqref{steady} to obtain
\begin{equation}\label{t7}
\begin{aligned}
&    \Gamma_{1,n}Aw_1+\Gamma_{2,n}Aw_n^{(2)}
+\alpha_n \Gamma_{2,n} \Bs(v,w_n^{(2)}) 
+\alpha_n \Gamma_{1,n}^2B(w_1,w_1)\\
&
+\alpha_n \Gamma_{1,n}\Gamma_{2,n} \Bs(w_1,w_n^{(2)}) 
+\alpha_n \Gamma_{2,n}^2B(w_n^{(2)},w_n^{(2)})=0 .  
\end{aligned}
\end{equation} 

\medskip
\textit{Proof of \ref{m1}.} Considering equation \eqref{t7} in each stated case, we have the following.
\begin{enumerate}[label=\rnum]
    \item If $\Gamma_{1,n}\succ \alpha_n \Gamma_{2,n},\alpha_n \Gamma_{1,n}^2$, then $Aw_1=0$, so we immediately have $w_1=0$.
    \item If $\Gamma_{1,n}\sim \alpha_n \Gamma_{1,n}^2 \succ \alpha_n \Gamma_{2,n}$, then 
    $$Aw_1 +\mu B(w_1,w_1)=0 \text{ for some } \mu> 0.$$
    Thus taking the product in $V'\times V$ with $w_1$ and applying \eqref{BV}, we find that, again $w_1=0$.
\end{enumerate}
Therefore, these two cases induce the degenerate expansion \eqref{degen2}. 

\medskip
\textit{Proof of \ref{m2}.} Dividing equation \eqref{t7} by $\alpha_n \Gamma_{2,n}$ and passing $n\to\infty$, we obtain $\Bs(v,w_2)=0.$

\medskip
One can prove the remaining items \ref{m3}--\ref{m6} similarly by considering the largest sequence among $(\Gamma_{1,n})_{n=1}^\infty$, $(\alpha_n \Gamma_{2,n} )_{n=1}^\infty$ and $(\alpha_n \Gamma_{1,n}^2)_{n=1}^\infty$ in equation \eqref{t7}. We omit the details.
\end{proof}

\subsection{The 2D no-slip case and 3D case}\label{3Dappln}
Consider 2D no-slip case or 3D case for both no-slip or periodic boundary condition.
Taking the inner product in $H$ of \eqref{steady} with $v_n$ and applying \eqref{Bflip}, \eqref{Poincare}, we have
\begin{equation}\label{Vbound} 
\|v_n\|\le |g|.
\end{equation}
With this bound,  we have the following analogue of Theorem \ref{2Dxp}.

\begin{thm}\label{3Dxp}
One has the following.

\begin{enumerate}[label=\tnum]
    \item The statement in Theorem \ref{2Dxp}(i) holds true with the interval $(1/2,1)$ being replaced with $(0,1/2)$.

\item Let $s$ be any number in the interval $[0,1/2)$.
Then there is a subsequence of $(v_n)_{n=1}^\infty$ that has a  unitary expansion or degenerate expansion in $D(A^s)$.    
\end{enumerate}
\end{thm}
\begin{proof}
    The proof is the same as that of Theorem \ref{2Dxp} with $Z_*=V$ and \eqref{DAbound} being replaced with \eqref{Vbound}. 
    Note that the limit $v$ in \eqref{vlim1} exists but is not given by \eqref{vlim2}.
\end{proof}

\hilite{
\section{Summary and discussion}
This work demonstrates that any sequence of steady state solutions to the Navier-Stokes equations has a subsequence that can be expanded using compactly embedded nested spaces.  The sequence is tied to a particular parameter, the Grashof number, $G$, as $G\to\infty$, which has particular significance for turbulence.  In the 2D periodic case, the quadratic nonlinearity of the NSE can be exploited to categorize the sequence limit along with leading terms.  These categories are determined by the decay rates of the coefficients of the terms in the expansion and characterized by additional relations that the limit of the sequence and certain leading terms must satisfy.  It should be emphasized that the results here are rigorous.

There are still challenges for the NSE in the 2D no-slip and 3D cases.  
More precisely, if we fix a number $s\in[0,1/2)$, then by the virtue of Theorem \ref{3Dxp}(ii), there is a subsequence, still denoted by $(v_n)_{n=1}^\infty$, that has a  unitary expansion or degenerate expansion
\begin{equation}\label{Asx}
v_n\approx v+\sum_{k\in\mathcal N} \Gamma_{k,n}w_k \text{ in } D(A^s).    
\end{equation}
However, unlike Theorems \ref{thm1}, \ref{v0eg} and \ref{vAinvg} for the the 2D periodic case, here we are not able to write down equations for $v$ and $w_k$. The main reason is that   $B(\cdot,\cdot)$, as currently defined, is not a continuous bilinear form from $V\times D(A^{s})$ or $D(A^s)\times V$ or $D(A^s)\times D(A^s)$ to $V'$, for $s<1/2$.
One possibility for carrying out this program in 3D could be to consider regularized Navier--Stokes equations, such as the Leray-$\alpha$ model \cite{Lerayalpha}.

This program for intrinsic expansion could naturally lead in other directions.  The generality of the expansion method 
makes it readily applicable to other models with quadratic (even polynomial) nonlinearities. 
Among them are the Kuramoto--Sivashinsky equation (KSE) as well as the Rayleigh--B\'enard (RB) and magneto-hydrodynamic systems \cite{T97}.  For instance, the role of the parameter $G$ could be played by the domain length for the 1D periodic KSE, and by the Rayleigh number for the RB system.
Finally, there is the question of extending this approach to  sequences of time dependent solutions.  
}
\appendix

\section{}

\begin{proof}[Proof of Proposition \ref{unique}]
First, the vector $v$ is determined uniquely by the limit \eqref{vlim}.
By the virtue of Proposition \ref{exclusive}, the three cases in Definition \ref{uexp} are exclusive. 

Case \ref{d1}: It is obvious that $\mathcal N=\emptyset$.

Case \ref{d2}: Then we have the identity \eqref{finvn}. Assume $v_n$ has another  nontrivial, finite strict expansion
\begin{equation*}
    v_n= v+\sum_{k\in\mathcal N'} \Gamma'_{k,n} w'_k\text{ together with vectors } w_n'^{(k)}.
\end{equation*}
Denote $K_0=\min\{\max\mathcal N,\max \mathcal N'\}$.
Applying Lemma \ref{preuniq} to $K:=K_0$, we have 
\begin{equation}\label{utemp}
    w_k=w_k',\quad \Gamma_{k,n}=\Gamma'_{k,n}, \quad w_n^{(k)}=w_n'^{(k)} \text{ for all }1\le k\le K_0, n\ge 1. 
\end{equation}

If $\max\mathcal N'>K_0$,  then, as in the case of \eqref{wcontra}, one has for $n\ge 1$ that 
\begin{align*} 
0=\left\|v_n-v-\sum_{k=1}^{K_0} \Gamma_{k,n}w_k\right\|_{Z_k}
&=\left\|v_n-v-\sum_{k=1}^{K_0} \Gamma_{k,n}'w_k'\right\|_{Z_k}\\
&=\Gamma_{K_0+1,n}'\|w_n'^{(K_0+1)}\|_{Z_{K_0}}=\Gamma'_{K_0+1,n}>0,
\end{align*}
which is a contradiction.  If $\max\mathcal N>K_0$,  then, similarly, 
\begin{align*} 
0
=\left\|v_n-v-\sum_{k=1}^{K_0} \Gamma_{k,n}'w_k'\right\|_{Z_k}
&=\left\|v_n-v-\sum_{k=1}^{K_0} \Gamma_{k,n}w_k\right\|_{Z_k}\\
&=\Gamma_{K_0+1,n}\|w_n^{(K_0+1)}\|_{Z_{K_0}}=\Gamma_{K_0+1,n}>0,
\end{align*}
which is a contradiction again.
Thus, $\mathcal N=\mathcal N'$  and, thanks to \eqref{utemp},  the vectors $w_k$, $w_n^{(k)}$ and numbers $\Gamma_{k,n}$ in \eqref{vsum} are determined uniquely for $k\in \mathcal N$ and $n\ge 1$. 

Case \ref{d3}:  Clearly, $\mathcal N=\mathbb N$. For each $K\ge 1$, we apply Lemma \ref{preuniq} to find that
$w_k$, $w_n^{(k)}$  and $\Gamma_{k,n}$ in \eqref{vsum} are determined uniquely for $1\le k\le K$ and $n\ge \bar N_k$.
Since $K$ is arbitrary, these properties then hold true for all $k\in\mathbb N$.
\end{proof}

The following additional types of asymptotic expansions arise in the proofs below.

\begin{defn}\label{refinex2}
Let  $\mathcal Z=(Z_k)_{k=0}^\infty$ be as in Definition \ref{uexp} and $(v_n)_{n=1}^\infty$ be a sequence in $Z_0$. 
 We say $(v_n)_{n=1}^\infty$ has a \emph{pre-unitary expansion}, respectively, \emph{pre-degenerate expansion} if 
it satisfies \ref{uni}, respectively, \ref{degenx} with the condition $ \|w_k\|_{Z_k}=1$ in \eqref{wkone}, respectively, \eqref{wktwo}, being  replaced with $w_k\ne 0$.
\end{defn}

\begin{proof}[Proof of Theorem \ref{refinethm}]
 The proof is divided into three steps.
 
\medskip   \emph{Step 1: The procedure to remove one zero term.} Denote $w_0=v$, $\Gamma_{0,n}=1$. Assume there is an integer $s\ge 2$  such that 
   $w_{s-1}=0$ and $w_s\ne 0$. The following shows that we can remove the term $\Gamma_{s-1,n}w_{s-1}$.
   
We have $\Gamma_{s,n}\to0$ when $s=2$,  and $\Gamma_{s,n}/\Gamma_{s-2,n}\to 0$ as $n\to \infty$ when $s>2$.
Then 
$$\sum_{k=0}^{s}\Gamma_{k,n}w_k
=\sum_{k=0}^{s-2}\Gamma_{k,n}w_k+\Gamma_{s,n}  w_s,$$
$$v_n=\sum_{k=0}^{s-1}\Gamma_{k,n}w_k+\Gamma_{s,n}  w_n^{(s)}
=\sum_{k=0}^{s-2}\Gamma_{k,n}w_k+\Gamma_{s,n}  w_n^{(s)},$$
and, in the case $w_{m}$ exists with $m\ge s+1$,
$$v_n=\sum_{k=0}^{m-1}\Gamma_{k,n}w_k+\Gamma_{m,n}  w_n^{(m)}
=\sum_{k=0}^{s-2}\Gamma_{k,n}w_k +\sum_{k=s}^{m-1}\Gamma_{k,n}w_k+\Gamma_{m,n}  w_n^{(m)}.$$
Moreover,
$$\Gamma_{s,n}w_n^{(s)}=v_n-\sum_{k=0}^{s-1}\Gamma_{k,n}w_k
=v_n-\sum_{k=0}^{s-2}\Gamma_{k,n}w_k=\Gamma_{s-1,n}w_n^{(s-1)}.$$
Thus,
$$w_n^{(s)}=\frac{\Gamma_{s-1,n}}{\Gamma_{s,n}}w_n^{(s-1)}\in Z_{s-2}.$$

Let $\tilde{\mathcal Z}$ be $\mathcal Z$ removing $Z_{s-1}$, then the expression \eqref{vsum} without the term $\Gamma_{s-1,n}w_{s-1}$ is a relaxed expansion of $(v_n)_{n=1}^\infty$ in $\tilde{\mathcal Z}$. 

\medskip    \emph{Step 2: Removing necessary zero terms.} Consider three cases in Definition \ref{uexp} for the relaxed expansion \eqref{vsum}.

\medskip\noindent
    \textit{Case \ref{d1}.} Then, obviously, $v_n=v$ is the  trivial unitary expansion in $\tilde{\mathcal Z}=\mathcal Z$.
    
\medskip\noindent
    \textit{Case \ref{d2}.} In this case, we have from \eqref{wKcond} that $w_K\ne 0$.     
    If needed, we can utilize the procedure in Step 1 finitely many times to remove all zero terms $\Gamma_{k,n}w_k$, starting from $k=K-1$ down to $k=1$. The final sequence $\tilde {\mathcal Z}$ is $\mathcal Z$ removing all $Z_k$ for $1\le k\le K$ with $w_k=0$ in \eqref{vsum}. 
The result is a nontrivial, finite relaxed expansion in $\tilde {\mathcal Z}$ with all nonzero $w_k$, that is, a finite pre-unitary expansion in $\tilde{\mathcal Z}$. The new integer $K$ is denoted by $\tilde K$.

\medskip\noindent
    \textit{Case \ref{d3}.} We consider three scenarios.

    \medskip
    \textit{Scenario IIIa:   $w_k=0$ for all $k\ge 1$.} Then $(v_n)_{n=1}^\infty$ has a degenerate expansion in $\tilde {\mathcal Z}={\mathcal Z}$ according to Definition \ref{refinex}\ref{degenx}, case (a).

    \medskip
    \textit{Scenario IIIb: there is an  integer $N\ge 1$ such that $w_k=0$ for all $k>N$ and $w_N\ne 0$.} 
    Note that such an integer $N$ is unique.
    Using the procedure in Step 1, we remove finitely many zero terms $\Gamma_{k,n}w_k$ for $1\le k< N$.
    The sequence $\tilde {\mathcal Z}$ is $\mathcal Z$ removing all $Z_k$ for $1\le k< K$ with $w_k=0$ in \eqref{vsum}.
    We obtain an infinite relaxed expansion in $\tilde {\mathcal Z}$ with  $w_k\ne 0$ for all $1\le k\le \tilde N$ and  $w_k=0$ for all $k>\tilde N$, where $\tilde N$ is a positive integer. Thus, we obtain a pre-degenerate expansion in $\tilde{\mathcal Z}$.

    \medskip
    \textit{Scenario IIIc: The sequence $(w_k)_{k=1}^\infty$ is not eventually zero.} Suppose there is a subsequence $(w_{k_j})_{j=1}^\infty$ of $(w_k)_{k=1}^\infty$ such that $w_{k_j}\ne 0$ for all $j\ge 1$.
    Applying the procedure in Step 1, we remove finitely many zero terms between $v$ and $\Gamma_{k_1,n}w_{k_1}$, and for each $j\ge 1$, remove finitely many zero terms between 
    $\Gamma_{k_j,n}w_{k_j}$ and $\Gamma_{k_{j+1},n}w_{k_{j+1}}$.
    It results in a sequence $\tilde {\mathcal Z}$ which is $\mathcal Z$ removing all $Z_k$ for $k\ge 1$ with $w_k=0$ in \eqref{vsum}, and an infinite relaxed expansion in $\tilde{\mathcal Z}$ with $w_k\ne 0$ for all $k\ge 1$. That is, we obtain an infinite pre-unitary expansion in $\tilde{\mathcal Z}$.

\medskip    \emph{Step 3: Normalizing nonzero terms.} 
 It remains to treat Case \ref{d2} and Scenarios IIIb and IIIc in Step 2 above.
We will convert the pre-unitary/pre-degenerate expansions to unitary/degenerate expansions.

\medskip\noindent
    \textit{Case \ref{d2}.} We normalize by defining, for $1\le k\le \tilde K$,
\begin{equation}\label{normalize}
    \hat w_k=w_k/\|w_k\|_{Z_k}, \ 
\hat \Gamma_{k,n}=\Gamma_{k,n}\|w_k\|_{Z_k},\ 
\hat w_n^{(k)}=w_n^{(k)}/\|w_k\|_{Z_k}.
\end{equation}
Then one has 
\begin{equation}
    \label{unchange}
\|\hat w_k\|_{Z_k}=1, \ 
\hat \Gamma_{k,n}\hat w_k= \Gamma_{k,n} w_k, \ 
\hat \Gamma_{k,n}\hat w_n^{(k)}=\Gamma_{k,n}w_n^{(k)}.
\end{equation}
Hence,
$$v_n=v+\sum_{j=1}^{k-1}  \Gamma_{k,n} w_k+\Gamma_{k,n}w_n^{(k)}
=v+\sum_{j=1}^{k-1} \hat \Gamma_{k,n}\hat  w_k+\hat \Gamma_{k,n}\hat w_n^{(k)}.$$
We also have the following limit in $Z_k$
$$\lim_{n\to\infty }\hat w_n^{(k)}=\frac{w_k}{\|w_k\|_{Z_k}}=\hat w_k.$$
Therefore, we obtain a nontrivial, finite unitary expansion 
$$v_n\approx v+\sum_{k=1}^{\tilde K} \hat \Gamma_{k,n}\hat  w_k \text{ in }\tilde{\mathcal Z}.$$

\medskip\noindent
    \textit{Scenario IIIb.} We normalize by \eqref{normalize} for $1\le k \le \tilde N$, and by
    \begin{equation*}
    \hat w_k=w_k=0, \ 
\hat \Gamma_{k,n}=\Gamma_{k,n},\ 
\hat w_n^{(k)}=w_n^{(k)} \text{ for } k>\tilde N.
\end{equation*}
    Then we obtain a degenerate expansion 
    \begin{equation}\label{vhatsum}
        v_n\approx v+\sum_{k=1}^\infty \hat \Gamma_{k,n}\hat  w_k \text{ in }\tilde{\mathcal Z},
    \end{equation} 
according to Definition \ref{refinex}\ref{degenx}, case (b).

\medskip\noindent
    \textit{Scenario IIIc.} We normalize by \eqref{normalize} for all $k\ge 1$. Then we obtain an infinite unitary expansion in $\tilde {\mathcal Z}$ of the same form as \eqref{vhatsum}.
\end{proof}

\begin{proof}[Proof of Corollary \ref{maincor}]
Let $(v_n)_{n=1}^\infty$ be a sequence in $Z_0$ that converges to $v\in Z_0$.
By the virtue of Theorem \ref{mainlem}, there is a subsequence, still denoted by $(v_n)_{n=1}^\infty$, such that it has a strict expansion  \eqref{vsum} in $\mathcal Z$.

\ref{caseq} By property \eqref{srex} and Theorem \ref{refinethm}, there exists a subsequence $\tilde {\mathcal Z}$ of $\mathcal Z$ such that $(v_n)_{n=1}^\infty$ has a unitary or degenerate expansion 
\begin{equation}\label{xNtil}
    v_n\approx v+\sum_{k\in\tilde{ \mathcal N}}\Gamma_{k,n}w_k \text{ in }\tilde {\mathcal Z}.
\end{equation}

\ref{casingle} Set 
$    \hat Z_k=Z_\infty$ for all $k\ge 1$, and $\hat{\mathcal Z}=(\hat Z_k)_{k=0}^\infty= (Z_\infty)_{k=0}^\infty$.
Then  \eqref{xNtil} is also a corresponding pre-unitary expansion or pre-degenerate expansion of $(v_n)_{n=1}^\infty$ in $\hat{\mathcal Z}$.
Applying the normalizing process in Step 3 in the proof of Theorem \ref{refinethm} to this expansion in $\hat{\mathcal Z}$, we obtain  the following  corresponding unitary or degenerate expansion 
\begin{equation}\label{xZZ}
    v_n\approx v+\sum_{k\in \tilde{\mathcal N}}\hat\Gamma_{k,n}\hat w_k  \text{ in }\hat{\mathcal Z}, 
\end{equation}
which, in fact, is a unitary or degenerate expansion in $Z_\infty$.
\end{proof}

\begin{obs}\label{degensing}
Let the subsequence $(v_n)_{n=1}^\infty$ in the proof of Corollary \ref{maincor} have two degenerate expansions \eqref{xNtil} and \eqref{xZZ}. 
Thanks to the second identity in \eqref{unchange}, the remainder $R_{N,n}$, see \eqref{Rmn}, is the same for both  \eqref{xNtil} and \eqref{xZZ}. 
Then it follows as in \eqref{denremest} that 
$$\lim_{n\to\infty} \frac{\|R_{N,n}\|_{Z_\infty}}{\Gamma_{k,n}}=0\text{ for any }k>N.$$      
\end{obs}

\medskip
\noindent\textbf{Data availability statement.} 
No new data were created or analysed in this study.

\medskip
\noindent\textbf{Acknowledgment.}
MSJ was supported in part by Simons Foundation Grant MP-TSM-00002337. 

\medskip
\noindent\textbf{Conflict of interest.}
There are no conflicts of interests.

\bibliographystyle{abbrv}
\bibliography{HJ1cite.bib}

\end{document}